\documentclass{article}

 \usepackage[letterpaper,top=2cm,bottom=2cm,left=3cm,right=3cm,marginparwidth=1.75cm]{geometry}

\usepackage{amsmath}
\usepackage{amsthm}
\usepackage{aliascnt}
\usepackage{mathtools}
\usepackage{algorithm}
\usepackage{algpseudocode}
\usepackage{float}
\usepackage{booktabs}
\usepackage{graphicx}
\usepackage{adjustbox}
\usepackage[colorlinks=true, allcolors=blue]{hyperref}
\usepackage{cleveref}
\usepackage{amssymb}
\usepackage{stmaryrd}
\usepackage{xcolor}
\usepackage{tikz}

\newtheorem{theorem}{Theorem}
\newaliascnt{lemma}{theorem}
\newtheorem{lemma}[lemma]{Lemma}
\aliascntresetthe{lemma}
\newaliascnt{corollary}{theorem}

\aliascntresetthe{corollary}

\newtheorem{proposition}{Proposition}

\crefname{theorem}{Theorem}{Theorems}
\Crefname{theorem}{Theorem}{Theorems}
\crefname{lemma}{Lemma}{Lemmas}
\Crefname{lemma}{Lemma}{Lemmas}
\crefname{corollary}{Corollary}{Corollaries}
\Crefname{corollary}{Corollary}{Corollaries}
\crefname{definition}{Definition}{Definitions}
\Crefname{definition}{Definition}{Definitions}
\crefname{proposition}{Proposition}{Propositions}
\Crefname{proposition}{Proposition}{Propositions}
\crefname{prop}{Proposition}{Propositions}
\Crefname{prop}{Proposition}{Propositions}

\DeclareMathOperator{\tr}{tr}

\DeclareMathOperator{\diag}{diag}

\DeclareMathOperator{\argmax}{argmax}
\DeclarePairedDelimiter{\norm}{\|}{\|}
\newcommand{\mat}[1]{\mathbf{#1}}

\newcommand{\lowrank}[1]{\llbracket #1 \rrbracket}

\newcommand{\I}{\mathcal{I}}
\newcommand{\J}{\mathcal{J}}
\newcommand{\prob}{\mathbb{P}}

\newcommand{\Id}{\mathbf{I}}
\newcommand{\bm}[1]{\boldsymbol{#1}}
\renewcommand{\vec}[1]{\bm{#1}}
\newcommand{\EX}{\mathbb{E}}

\title{Low-Rank Approximation by Randomly Pivoted LU}
\author{
Marc Aurèle Gilles\thanks{Department of Mathematics, Princeton University, Princeton, NJ 08544, USA (gilles@princeton.edu)}
\and
Heather Wilber\thanks{Applied Mathematics Department, University of Washington, Seattle, WA 98195, USA (hdw27@uw.edu)}
}

\begin{document}
\setcounter{footnote}{1}
\maketitle

\begin{abstract}
The low-rank approximation properties of Randomly Pivoted LU (RPLU), a variant of Gaussian elimination where pivots are sampled proportional to the squared entries of the Schur complement, are analyzed. 
It is shown that the RPLU iterates converge geometrically in expectation for matrices with rapidly decaying singular values. RPLU outperforms existing low-rank approximation algorithms in two settings: first, when memory is limited, RPLU can be implemented with $\mathcal{O}(k^2 + m + n)$ storage and $\mathcal{O}( k(m + n)+  k\mathcal{M}(\mat{A}) + k^3)$ operations, where $\mathcal{M}(\mat{A})$ is the cost of a matvec with $\mat{A}\in\mathbb{C}^{n\times m}$ or its adjoint, for a rank-$k$ approximation. Second, when the matrix and its Schur complements share exploitable structure, such as for Cauchy-like matrices.
The efficacy of RPLU is illustrated with several examples, including applications in rational approximation and solving large linear systems on GPUs.  
\end{abstract}

\section{Introduction} \label{sec:intro}

Gaussian elimination (GE) with pivoting is the standard algorithm for solving systems of linear equations.
Given  $\mat{A} \in \mathbb{C}^{n \times n}$, it constructs a pivoted LU factorization  $\mat{P}\mat{A}\mat{Q} = \mat{{L}}\mat{{U}}$, where $\mat{{L}}$ and $\mat{{U}}$ are lower and upper triangular matrices respectively, and $\mat{P}$ and $\mat{Q}$ are permutation matrices.
The numerical stability of the computation depends on the strategy applied to sequentially choose the permutation matrices. These are called \textit{pivoting strategies}. The strategy used most widely in practice is partial pivoting, which selects $\mat{P}$ relatively cheaply using only the leading entries of rows of the residual matrix, and sets $\mat{Q} = \mat{I}$. However, complete pivoting, which permutes both rows and columns at each step, is also a popular choice and has stronger theoretical stability guarantees~\cite{wilkinson1961error,trefethen2022numerical}.

By stopping GE early, complete pivoting can also be used to produce structured low-rank approximations to $\mat{A}$. We set $\mat{\hat{A}} \approx \mat{\hat{L}}\mat{\hat{U}}$, where $\mat{\hat{L}} = \mat{P}^T\mat{L}_{:, 1:k}$, $\mat{\hat{U}} = \mat{U}_{1:k, :} \mat{Q}^T$, and $\mat{L}_{:, 1:k}$ and $\mat{U}_{1:k, :}$ are the first $k$ columns and rows of $\mat{L}$ and $\mat{U}$, respectively. Triangularity of the LU factors is no longer important in this setting. This idea is known in different contexts under many different names, including adaptive cross approximation~\cite{zhao2005adaptive}, GE with complete pivoting~\cite{townsend2015continuous}, the method of Geddes--Newton series expansions~\cite{chapman2003generalized}, and incomplete LU factorization~\cite{chan1997approximate}. In the special case where $\mat{A}$ is square and positive semidefinite (PSD), related algorithms making use of Cholesky decompositions are known as incomplete or partial Cholesky factorizations~\cite{rpchol}.

\begin{figure}[t]
    \centering
    \includegraphics[width=0.7\textwidth]{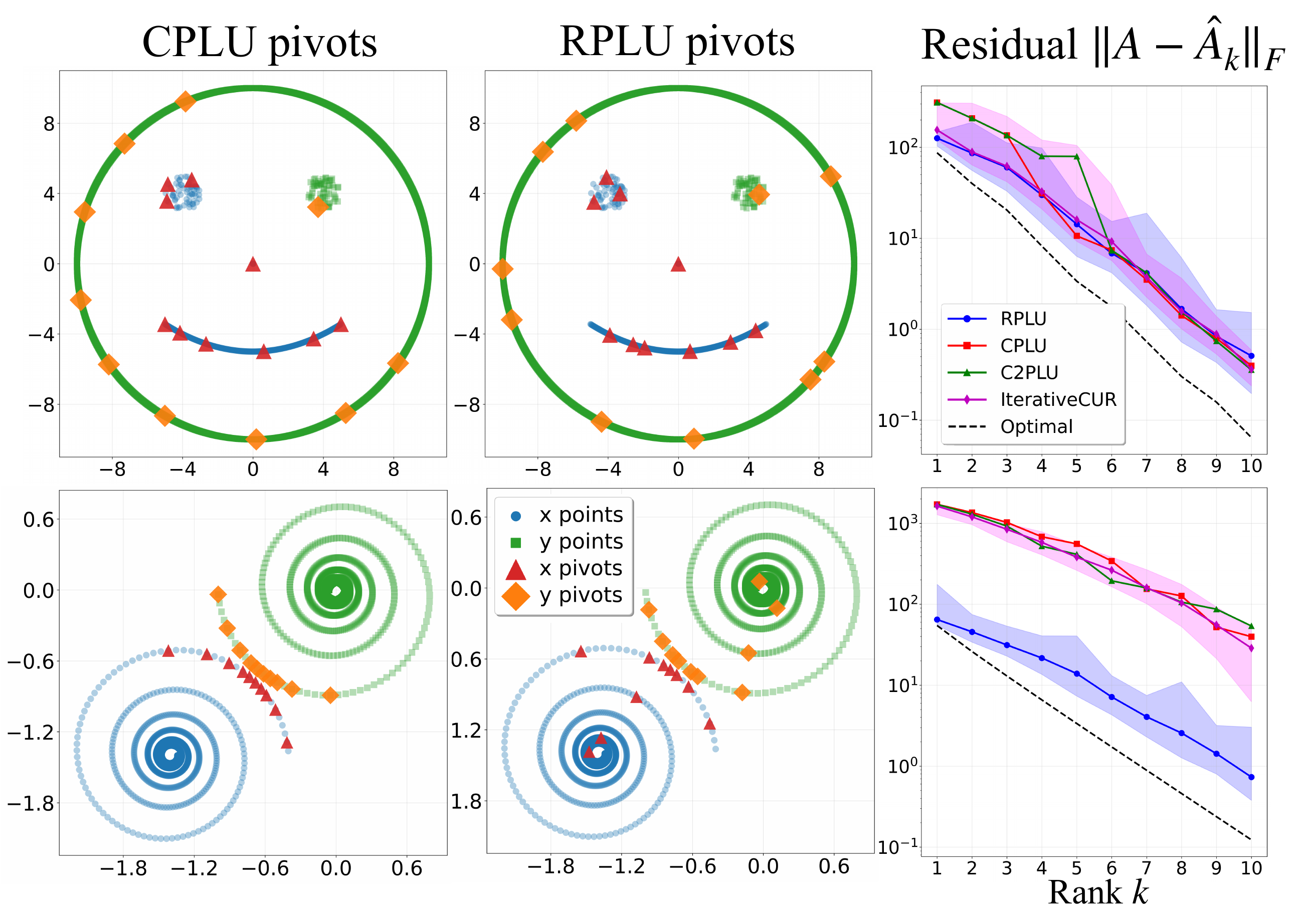}
    \caption{Performance of RPLU and three greedy schemes CPLU, C2PLU and IterativeCUR~\cite{pritchard2025fast} ($b=5$) on a Cauchy matrix $A_{ij} = 1 / (x_i - y_j)$ for points $\{x_i\}$ (blue) and $\{y_j\}$ (green) forming a smiley (top) and two spirals (bottom). Left: pivots selected by CPLU. Middle: pivots from one run of RPLU. Right: approximation error vs.\ rank $k$, with mean over 30 runs and min/max as a shaded region for RPLU and IterativeCUR. In the spiral example, greedy schemes select only ``outlier" points, whereas RPLU also samples from the cluster, yielding lower error. Similar behavior is observed for PSD matrices~\cite{rpchol}.}
    \label{fig:rplu_intro}
\end{figure}

A partial LU factorization of $\mat{A} \in \mathbb{C}^{n \times m}$ produces the following sequence of matrices:
\begin{align}
\mat{A}^{(0)} &= \mat{A}, \quad \mat{A}^{(k)} = \mat{A}^{(k-1)} - \frac{1}{\mat{A}^{(k-1)}_{i_k, j_k}} \mat{A}^{(k-1)}_{:, j_k} \mat{A}^{(k-1)}_{i_k,:} \label{eq:residual_update} \\ 
\mat{\hat{A}}^{(0)} &= \mat{0}, \quad \mat{\hat{A}}^{(k)} = \mat{\hat{A}}^{(k-1)} + \frac{1}{\mat{A}^{(k-1)}_{i_k, j_k}} \mat{A}^{(k-1)}_{:, j_k} \mat{A}^{(k-1)}_{i_k,:} \label{eq:approximation_update}
\end{align}
where $\mat{A}^{(k)}$ is the residual matrix after $k$ steps of GE and $\mat{\hat{A}}^{(k)}$ is a rank-$k$ approximation of $\mat{A}$. Here $\mat{A}^{(k)}_{i, :}$ and $\mat{A}^{(k)}_{:, j}$ denote the $i$-th row and $j$-th column of $\mat{A}^{(k)}$. We refer to the pair $(i_k, j_k)$ as the \textit{pivot} chosen at step $k$.
Complete pivoting selects $(i_k, j_k) = \argmax_{i,j} |\mat{A}^{(k-1)}_{i,j}|$, the largest absolute entry of the residual matrix. We call this ``completely pivoted LU'' (CPLU).

In~\cite{rpchol}, the authors showed that for PSD matrices, injecting randomness into pivot selection can substantially improve approximation quality over greedy schemes in some settings. In this paper we consider an analogous algorithm for the nonsymmetric setting. We apply a randomized version of complete pivoting, where the pivot $(i_k, j_k)$ at step $k$ is sampled according to the probability distribution:
\begin{equation}
\prob(i, j\mid \mat{A}^{(k-1)}) \;=\; \frac{|\mat{A}^{(k-1)}_{i, j}|^2}{\|\mat{A}^{(k-1)}\|_F^2}.
\end{equation}
We call the resulting method ``Randomly Pivoted LU'' (RPLU), analogous to the Randomly Pivoted Cholesky (RPCholesky) method in~\cite{rpchol}. A basic implementation is shown in \Cref{alg:rplu}.

\begin{algorithm}[h!]
  \caption{Randomly pivoted LU (Na\"ive implementation)}
  \label{alg:rplu}
  \textbf{Input:} Matrix $\mat{A} \in \mathbb{C}^{n \times m }$; desired rank $r$.\\
  \textbf{Output:} Matrices $\mat{L} \in \mathbb{C}^{n\times r}$ and  $\mat{U} \in \mathbb{C}^{r\times m}$ defining the approximation $\hat{\mat{A}} = \mat{L}\mat{U}$
  \begin{algorithmic}
    \State Initialize $\mat{L} \leftarrow \mat{0}_{n\times r}$, $\mat{U} \leftarrow \mat{0}_{r\times m}$, $\mat{A}^{(0)} \leftarrow \mat{A}$
    \For{$k = 1$ to $r$}
      \State Sample pivot $\bigl(i_k, j_k\bigr) \sim \bigl|\mat{A}^{(k-1)}\bigr|^2 \big/ \|\mat{A}^{(k-1)}\|_F^2$  \Comment{Sample pivot}
      \State $\mat{L}(:,k) \leftarrow \mat{A}^{(k-1)}(:,j_k)\big/\mat{A}^{(k-1)}(i_k,j_k)$
      \State $\mat{U}(k,:) \leftarrow \mat{A}^{(k-1)}(i_k,:)$
      \State $\mat{A}^{(k)} \leftarrow \mat{A}^{(k-1)} - \mat{L}(:,k)\,\mat{U}(k,:)$  \Comment{Update residual}
    \EndFor
  \end{algorithmic}
\end{algorithm}

Our main theoretical result (see~\cref{theorem:doubling}) shows that RPLU will return a good rank-$k$
approximation to $\mat{A}$ if its singular values decay at the rate $\mathcal{O}(\rho^k)$ with $\rho < 1/2$. This is an improvement on deterministic complete pivoting, which requires $\rho < 1/4$ in a worst-case scenario~\cite{cortinovis2020maximum}. As shown in \Cref{fig:rplu_intro}, there are settings where RPLU decisively outperforms CPLU and other greedy pivoting methods. 
 
Our second contribution addresses efficient implementation. 
Unlike the PSD case, where one only needs to consider diagonal entries for pivoting, iterations of randomly (or completely) pivoted LU are expensive: each pivot selection and residual update touches the full matrix, costing $\mathcal{O}(nm)$ operations. However, we show that RPLU can be implemented much more efficiently for many structured matrix classes. For Cauchy-like matrices, we employ fast Schur complement updates via displacement structure~\cite{boros2002pivoting,gohberg1995fast,heinig2012inversion, pan2000superfast}, along with Barnes-Hut-style row norm estimates~\cite{barnes1986hierarchical}. For matrices with fast matvecs (e.g., Toeplitz, sparse, or kernel matrices), we introduce a memory-efficient variant achieving $\mathcal{O}(k(n+m) + k^3 + k\mathcal{M}(\mat{A}) + k\mathcal{M}(\mat{A}^T))$ operations and $\mathcal{O}(k^2 + n + m)$ storage, where $\mathcal{M}(\mat{A})$ is the matvec cost. This variant is especially effective on memory-constrained hardware such as GPUs: in~\Cref{sec:practical_algorithm}, we construct rank-$1000$ approximations to matrices with $30$ million rows and columns in minutes, which would require half a terabyte for dense storage.

 The remainder of this paper is organized as follows. \Cref{sec:background} reviews background on low-rank decompositions. \Cref{sec:results} presents convergence results for RPLU. \Cref{sec:practical_algorithm} describes an efficient implementation for large structured matrices with numerical experiments. \Cref{sec:cauchy} presents implementation and numerical results for Cauchy-like matrices.

\section{Background} \label{sec:background}

\subsection{Notation}
Let $\mat{A}\in\mathbb{C}^{n\times m}$ with entries $a_{i,j}$ for $1 \leq i \leq n$ and $1 \leq j \leq m$. Without loss of generality, we assume $n \leq m$. We use MATLAB-style notation where $\mat{A}_{:,j}$ denotes the $j$-th column, $\mat{A}_{i,:}$ denotes the $i$-th row, and $\mat{A}_{\I,\J}$ denotes the submatrix with rows $\I \subset \{1, \ldots, n\}$ and columns $\J \subset \{1, \ldots, m\}$.
Let the singular values of $\mat{A}$ be given by $\sigma_1\ge\cdots\ge\sigma_{n}\ge0$. We define the following norms:
\begin{align*}
\|\mat{A}\|_F = \sqrt{\sum_{i=1}^n \sigma_i^2},  \qquad
\|\mat{A}\|_{\star} = \sum_{i=1}^n \sigma_i, \qquad
\|\mat{A}\|_{2} = \sigma_1.
\end{align*}
We denote a PSD matrix $\mat{A} \in \mathbb{C}^{n\times n}$ by $\mat{A} \succeq 0$, and use $\preceq$ as the partial order on $\mathbb{C}^{n\times n}$ defined by $\mat{A} \preceq \mat{B}$ if $\mat{A} - \mat{B} \succeq 0$. We denote the trace of a matrix $\mat{A}$ by $\tr(\mat{A}) = \sum_{i=1}^n a_{ii}$, and recall that if $\mat{A} \succeq 0$, then $\tr(\mat{A}) = \|\mat{A}\|_{\star} $.

\subsection{Optimal low-rank approximations}
The singular value decomposition of $\mat{A}$ is given by $\mat{A} = \mat{U}\mat{\Sigma}\mat{V}^*$ where $\mat{U}\in\mathbb{C}^{n\times n}$ and $\mat{V}\in\mathbb{C}^{m\times m}$ are unitary and $\mat{\Sigma}\in\mathbb{R}^{n\times m}$ is diagonal with entries $\Sigma_{jj} = \sigma_j$ for $1 \leq j \leq n$. The matrix $\lowrank{\mat{A}}_{k} = \mat{U}_{:\,,\,1 : k} \mat{\Sigma}_{1 : k, 1 : k} \mat{V}_{:\,,\,1 : k}^*$ is called the rank-$k$ truncated SVD of $\mat{A}$. 
The Eckart-Young-Mirsky theorem states that  $\lowrank{\mat{A}}_{k}$ is optimal among all rank-$k$ matrices in any unitarily invariant norm. In particular,
\begin{align*}
\min_{\mathrm{rank}(\mat{X})\le k} \|\mat{A}-\mat{X}\|_F \,=\, \|\mat{A}-\lowrank{\mat{A}}_{k}\|_F \,&=\, \Bigl(\sum_{j>k}\sigma_j^2\Bigr)^{1/2}, \\
\min_{\mathrm{rank}(\mat{X})\le k} \|\mat{A}-\mat{X}\|_2 \,=\, \|\mat{A}-\lowrank{\mat{A}}_{k}\|_2 \,&=\, \sigma_{k+1}.
\end{align*}
While optimal, computing the truncated SVD is often prohibitively expensive for large problems, requiring $\mathcal{O}(knm)$ operations. A large body of work has developed randomized algorithms that approximate the rank-$k$ SVD more efficiently (see~\cite{halko2011finding} for a survey). Among these, randomized SVD has become a gold-standard for approximately computing the SVD of a matrix due to some favorable properties: (1) improved complexity, e.g. $\mathcal{O}( \log(k) nm)$ for a general matrix, and better for structured matrices; (2) strong theoretical guarantees, e.g. with high probability, the error is bounded by a constant factor of the optimal error; (3) use of level 3 BLAS operations, making it highly efficient in practice and particularly well-suited for modern hardware (e.g. GPUs).

\subsection{Feature-preserving decompositions}

In many applications, it is important for either interpretability or numerical efficiency to have a low-rank approximation to $\mat{A}$ that preserves certain features of the matrix, such as sparsity, non-negativity~\cite{mahoney2009cur}, or a meaningful connection to variables and features linked to column or row indices. In such cases, it is natural to consider low-rank factorizations comprised of subsets of the rows and columns of $\mat{A}$, such as CUR decompositions or interpolative decompositions~\cite{cheng2005compression,stewart1999four,berry2005algorithm,goreinov1997theory,goreinov2001maximal,mahoney2009cur}.

CUR decompositions have the form $\mat{A} \approx \mat{C}\mat{U}\mat{R}$, where $\mat{C} = \mat{A}_{:,\J} \in \mathbb{C}^{n\times k}$ and $\mat{R} = \mat{A}_{\I,:} \in \mathbb{C}^{k\times m}$ are submatrices formed by selecting column and row index sets $\J \subset \{1, \ldots, m\}$ and $\I \subset \{1, \ldots, n\}$ with $|\I| = |\J| = k$ (though some work allows $|\I| \neq |\J|$), and $\mat{U} \in \mathbb{C}^{k\times k}$ is a small core matrix. Given fixed index sets $\I$ and $\J$, 
several choices of $\mat{U}$ are possible. The two most common choices are $\mat{U} = \mat{R}^{\dagger}\mat{A}\mat{C}^{\dagger,*}$, which we refer to as \textit{projective} CUR, and $\mat{U} = \mat{A}(\I,\J)^{\dagger}$, which we refer to as \textit{interpolative} CUR.
The projective CUR is the optimal choice of $\mat{U}$ that minimizes the Frobenius norm error~\cite[section 4]{stewart1999four}~\cite[Proposition 10.1]{epperly2025make}, while the interpolative CUR agrees with the chosen rows and columns exactly, and interpolates the other entries.

Closely related to CUR decompositions are interpolative decompositions and other ``CUR-like'' factorizations that store
\(\mat{A} \approx \mat{Z}\mat{U}\mat{W}\), where the columns of \(\mat{Z}\) and the rows of \(\mat{W}\) span (approximately) the column space of \(\mat{C}\) and the row space of \(\mat{R}\), respectively. Examples include skeletonization~\cite{cheng2005compression}, stable CUR variants~\cite{epperly2025make,ekentaspectrum}, and incomplete LU and Cholesky factorizations.
The na\"{\i}ve version of RPLU shown in \Cref{alg:rplu}, as well as the standard implementation of CPLU, construct such CUR-like decompositions. However, as we show in \Cref{sec:practical_algorithm}, maintaining the approximation in true CUR form is sometimes preferable, especially when entries of \(\mat{A}\) can be accessed on demand without storing \(\mat{A}\) explicitly (e.g., for kernel matrices).

This raises the question of how to choose the index sets $\I$ and $\J$. A theoretically appealing but expensive approach is maximum-volume pivoting. Here, row and column subsets $\I$ and $\J$ are selected to globally maximize the volume $|\det(\mat{A}(\I, \J))|$. This strategy provides a strong theoretical guarantee on the approximation error for CUR decompositions~\cite{damle2024reveal,goreinov2001maximal}:
\[
\|\mat{A} - \mat{A}(\,:\,,\J) \mat{A}(\I,\J)^{\dagger} \mat{A}(\I,\,:\,)\|_{\max} \leq (k+1)\,\sigma_{k+1},
\]
\[
\|\mat{A} - \mat{A}(\,:\,,\J) \mat{A}(\I,\J)^{\dagger} \mat{A}(\I,\,:\,)\|_{2} \leq (1+5k\sqrt{mn})\,\sigma_{k+1},
\]
where $\|\cdot\|_{\max}$ denotes the entrywise maximum norm.
The latter bound shows that interpolative CUR decompositions can achieve errors comparable to the SVD, up to a modest factor depending on $k,m,n$. Unfortunately, maximum-volume pivoting is too expensive: it would require evaluating the determinant of all $\binom{n}{k}\binom{m}{k}$ candidate submatrices, which is exponential in $k$. Practical methods therefore trade such strong bounds for computational efficiency. Recent work in~\cite{damle2024reveal} introduces \emph{local} maximum-volume pivoting and shows it can achieve the second bound in polynomial time. Many practical CUR algorithms exist; the bulk fall into three fundamental classes:

\vspace{.3cm} 

\noindent \textbf{All-at-once strategies:} These methods select $k$ columns (and/or rows) in one step or sweep~\cite{drineas2006fast,deshpande2010efficient,wang2013improving}. More sophisticated variants sample using weighted distributions based on column norms or leverage scores~\cite{mahoney2009cur}. These methods are attractive when access to $\mat{A}$ is limited (e.g., one-pass or two-pass settings), but they are not the focus of this work.

\vspace{.3cm}

\noindent \textbf{One-at-a-time strategies:} These methods select columns (and/or rows) iteratively, typically using a residual-based heuristic. Sometimes they include limited ``backtracking'' or ``swapping'', which can make them more robust. They can achieve high accuracy for a given rank $k$, but typically require repeated access to $\mat{A}$. Many popular one-at-a-time strategies make use of QR and LU decompositions with various pivoting strategies~\cite{damle2024reveal,epperly2025make,sorensen2016deim,voronin2017efficient}.

\vspace{.3cm}

\noindent \textbf{Sketch-and-pivot strategies:} These methods begin by forming a sketch of $\mat{A}$ and then performing column (or row) subset selection on the sketch. For example, if $\mat{B} \approx \mat{\Omega}\mat{A}$ where $\mat{\Omega} \in \mathbb{C}^{k \times n}$ is drawn from a random distribution and $k=\Theta(r)$, then one can select columns of $\mat{B}$ and use the corresponding indices to form $\mat{C}$ from $\mat{A}$, see e.g.,~\cite{dong2023simpler,pritchard2025fast}. 

\vspace{.3cm}

Many other approaches combine and refine the strategies above. Our focus in this paper is settings where one-at-a-time methods, such as partial LU factorizations, are applicable and fast. We now describe the major one-at-a-time strategies most closely related to our work.  

\subsubsection{CUR by pivoted QR}

A classical route to CUR and interpolative decompositions is through a column-pivoted QR (CPQR) factorization.   Given \(\mat{A}\in\mathbb{C}^{n\times m}\) and a target rank \(k\), CPQR produces a permutation matrix \(\mat{\Pi}\) such that
\[
\mat{A}\mat{\Pi}
\;=\;
\mat{Q}
\begin{bmatrix}
\mat{R}_{11} & \mat{R}_{12}\\
0 & \mat{R}_{22}
\end{bmatrix},
\qquad
\mat{R}_{11}\in\mathbb{C}^{k\times k},
\]
where \(\mat{Q}\) is unitary and \(\mat{R}\) is upper triangular. Let \(\J\) denote the first \(k\) pivot columns (encoded by \(\mat{\Pi}\)). Dropping the trailing block yields a column interpolative decomposition
\[
\mat{A}
\;\approx\;
\mat{A}_{:,\J}\,\mat{T},
\qquad
\mat{T}
\;=\;
\begin{bmatrix}\mat{I}_k & \mat{R}_{11}^{-1}\mat{R}_{12}\end{bmatrix}\mat{\Pi}^{\top},
\]
with \(\|\mat{A}-\mat{A}_{:,\J}\mat{T}\|_2 = \|\mat{R}_{22}\|_2\).
Applying the same construction to \(\mat{A}^*\) selects a set of pivot rows \(\I\), and combining the resulting row and column subsets gives a CUR approximation \(\mat{A}\approx \mat{C}\mat{U}\mat{R}\) with \(\mat{C}=\mat{A}_{:,\J}\) and \(\mat{R}=\mat{A}_{\I,:}\); the core \(\mat{U}\) can then be chosen using either the projective or interpolative prescription described above. While CPQR employs a simple greedy pivoting heuristic first described by Businger and Golub in~\cite{businger1965linear}, many other variations of pivoted QR have been suggested and studied in the literature~\cite{chandrasekaran1994rank, damle2024reveal, gu1996efficient}. 

In principle CPQR can be numerically unstable for adversarially designed examples, but this is rarely an issue in practice.  Pivoted-QR based CUR methods are typically numerically robust, since the orthogonal transformations help control element growth. Moreover, memory-efficient implementations exist for sparse matrices that do not form \(\mat{Q}\) explicitly (``Q-less'' CPQR), maintaining only the permutation, a small triangular factor, and norm estimates~\cite{stewart1999four,berry2005algorithm}. See also~\cite{voronin2017efficient,epperly2025make}.

\subsubsection{LU with complete pivoting (CPLU)}

As discussed in \Cref{sec:intro}, stopping CPLU after $k$ steps yields the rank-$k$ approximation $\mat{\hat{A}}^{(k)}$ in~\cref{eq:approximation_update}. CPLU can also be modified to produce an \emph{interpolative} CUR decomposition in which $\mat{C}$ and $\mat{R}$ consist of columns and rows of $\mat{A}$, with index sets $\I$ and $\J$ given by the pivot sequence. This can be more memory efficient when entries of $\mat{A}$ can be accessed on demand.
Nonetheless, locating each pivot requires $\mathcal{O}(mn)$ work per step in general.
Despite its cost, the simplicity and effectiveness of complete pivoting has made it a popular technique for finding separable approximations of functions~\cite{townsend2013extension,zhao2005adaptive}, and it has received significant attention in the literature~\cite{cortinovis2020maximum,harbrecht2012low,townsend2015continuous}. In some cases, an approximate pivot can greatly accelerate the process~\cite{townsend2015continuous} with empirical evidence suggesting little loss in quality.
The sharpest known bound on the approximation error for LU with complete pivoting is~\cite[Theorem 6]{cortinovis2020maximum}:
\begin{equation} \label{eq:complete_pivoting_bound}
\|\mat{A} - \mat{\hat{A}}^{(k)}\|_{\max} \leq \rho_k 4^k \, \sigma_{k+1},
\end{equation}
where $\rho_k := \sup_{A} \{\|\mat{A}^{(k)}\|_{\max}\big/\|\mat{A}\|_{\max}\}$ is the \emph{growth factor}. The growth factor is subgeometric in $k$~\cite{wilkinson1961error,bisain2025new}. This bound shows that LU with complete pivoting will provably return a good low-rank approximation only if the singular values decay geometrically, i.e., $\sigma_k = \mathcal{O}(\rho^k)$ with rate $\rho < \frac{1}{4}$. Importantly, the exponential factor $4^k$ in the error bound is asymptotically tight~\cite[Rem.~3.3]{harbrecht2012low}\cite{higham1987survey}. In practice, however, this bound is often pessimistic, and LU with complete pivoting typically performs much better than~\cref{eq:complete_pivoting_bound} might suggest, though it can fail in the presence of outlier columns (see~\Cref{fig:rplu_intro} for an example).

\subsubsection{Complete 2-norm pivoting (C2PLU)}

A less-studied pivoting strategy relevant to our work is complete $2$-norm pivoting~\cite{melgaard2015gaussian}, which selects pivots from the residual $\mat{A}^{(k)}$ via
\begin{equation}
i_{k+1} = \argmax_i \| \mat{A}^{(k)}_{i,:} \|_2, \quad
j_{k+1} = \argmax_j | \mat{A}^{(k)}_{i,j} |.
\end{equation}
That is, at step $k\!+\!1$ the row pivot is the row of $\mat{A}^{(k)}$ with the largest $2$-norm, and the column pivot is the largest-magnitude entry in that row.
This strategy (and a sketched variant) was studied as an alternative to complete pivoting to improve the numerical stability of Gaussian elimination~\cite{melgaard2015gaussian}. An approximate version based on matrix sketching was proposed in~\cite{ekentaspectrum,chen2020efficient}, where pivots are selected from the sketch and a projective CUR decomposition is formed from the resulting row and column indices.

A closely related recent method is IterativeCUR~\cite{pritchard2025fast}, which combines sketching with partial pivoting to construct a CUR approximation iteratively.
IterativeCUR is also closely related to our CUR formulations of RPLU and C2PLU in~\Cref{sec:practical_algorithm}: for batch size $b=1$, it can be viewed as a sketched, approximate variant of C2PLU, where row-norm information is estimated using a small recycled sketch. In settings with outliers (e.g., \Cref{fig:rplu_intro}), this sketch-induced randomness does not appear to mitigate the outlier sensitivity that motivates RPLU. Nevertheless, the strong empirical performance of IterativeCUR on general matrices suggests that incorporating related sketching ideas could also improve performance in the low-memory and Cauchy-like settings considered in~\Cref{sec:practical_algorithm,sec:cauchy}.

To the best of our knowledge, no theoretical or empirical results are known on the quality of low-rank approximations produced by (non-approximate) complete $2$-norm pivoting.
However, as we show in~\Cref{sec:results}, C2PLU is a natural greedy, deterministic counterpart to RPLU and can be implemented in the same way and with the same cost as RPLU (unlike CPLU); in our experiments it often outperforms RPLU.

\subsubsection{Pivoted Cholesky}
When \(\mat{A}\) is PSD, the symmetry can be exploited by choosing $\I = \J$, reducing the CUR approximation to a Nyström approximation of the form $\mat{A} \approx \mat{A}(:,\I) \mat{A}(\I,\I)^{\dagger} \mat{A}(\I,:)$. The Nyström approximation is often computed via a partial Cholesky decomposition, the symmetric analogue of the rank-1 update in~\cref{eq:approximation_update}. Traditionally, pivots are chosen greedily by selecting the largest diagonal entry:
\[
j_k = i_k \;=\; \argmax_i \, \mat{A}^{(k)}_{ii},
\]
where \(\mat{A}^{(k)}\) denotes the residual after \(k\) updates in~\cref{eq:residual_update}. Despite the PSD structure, the worst-case approximation quality of this greedy choice essentially mirrors the nonsymmetric setting. The best known bound for this choice~\cite[Corollary 9]{cortinovis2020maximum} is:
\begin{equation} \label{eq:nystrom_bound}
\|\mat{A}-\mat{\hat{A}}^{(k)}\|_{\max} \;\le\; 4^{\,k}\,\sigma_{k+1},
\end{equation}
and the factor \(4^{k}\) is asymptotically tight (the asymptotic example in~\cite{higham1987survey} is PSD). Unlike the nonsymmetric case, pivots can be chosen very efficiently here since one only needs to examine the diagonal entries. This efficiency makes the Nyström method an important approach in many applications, including kernel regression.
Similarly to the nonsymmetric case, this greedy scheme typically performs much better than~\cref{eq:nystrom_bound} suggests but can fail in the presence of outliers (see~\Cref{fig:symmetric_pivoting} for an example).

\subsubsection{Randomly pivoted Cholesky (RPCholesky)}
The inspiration for the present work is the analysis of Randomly Pivoted Cholesky (RPCholesky)~\cite{rpchol} in the PSD setting, which selects the pivot from the diagonal at step \(k\) with probability
\[
\prob(i\mid \mat{A}^{(k-1)}) \;=\; \frac{\mat{A}^{(k-1)}_{ii}}{\tr\!\bigl(\mat{A}^{(k-1)}\bigr)}.
\]
The authors establish two key results. The first is a doubling lemma, in a similar spirit to bounds of pivoted Cholesky and CPLU:
\begin{lemma}[Error doubling, Lemma 5.5 in~\cite{rpchol}]\label{lem:rchol_doubling}
For every PSD matrix \(\mat{A}\), the residual of RPCholesky satisfies
\[
\mathbb{E}\,\tr\!\bigl(\mat{A} - \mat{\hat{A}}^{(k)}\bigr)
\;\le\;
2^{k}\,\tr\!\bigl(\mat{A}-\lowrank{\mat{A}}_{k}\bigr),
\qquad k\in\mathbb{N}.
\]
\end{lemma}

This lemma shows that the randomized pivoting improves upon greedy pivoting in the worst case (the factor of $4^k$ is reduced to $2^k$), and that the residual norm will decay geometrically when $\sigma_k = \mathcal{O}(\rho^{k})$ for $\rho < \frac{1}{2}$.

The other bound in~\cite{rpchol} concerns itself with an oversampled version, which compares the expected trace of the residual after $k$ steps with the trace of the best rank $r$ approximation, with $k \geq r$:
\begin{theorem}[simplified bound, Theorem 2.3 in~\cite{rpchol}]\label{theorem:rpchol_bound}
   Fix \(r\in\mathbb{N}\) and \(\varepsilon>0\), and let \(\mat{A}\) be PSD.
    The Nyström approximation \(\hat{\mat{A}}^{(k)}\) produced by RPCholesky attains the error bound
    
    \begin{equation}
    \mathbb{E}\,\tr\!\bigl(\mat{A} - \mat{\hat{A}}^{(k)}\bigr)
\;\le\;
(1+\varepsilon) \tr\!\bigl(\mat{A}-\lowrank{\mat{A}}_{r}\bigr),
    \end{equation} provided
    \[
    k \;\ge\; \frac{r}{\varepsilon} \;+\; r\,\log\!\Bigl(\frac{1}{\varepsilon\,\eta}\Bigr),
    \qquad
    \eta \coloneqq \frac{\tr(\mat{A}-\lowrank{\mat{A}}_{r})}{\tr(\mat{A})}.
    \]
\end{theorem}
This result provides a useful bound for a much larger class of matrices than~\cref{lem:rchol_doubling}, and shows that this relatively simple scheme will return an approximation comparable in quality to the best rank $r$ approximation after $k\geq r$ steps in many settings.

However, because $\eta$ increases with the rank $r$, this bound does not guarantee geometric decay of the residual, regardless of the singular value structure of the matrix $\mat{A}$; we come back to this issue in~\Cref{sec:one_step_bound}. While our results are most similar to RPCholesky, we highlight that other PSD pivoting strategies have been proposed and analyzed; see, e.g.,~\cite{belabbas2009spectral,gittens2011spectral, guruswami2012optimal, fornace2024column, steinerberger2024randomly}.

\section{Analysis of RPLU} \label{sec:results}
In this section, we analyze the low-rank approximation performance of RPLU. 
Our analysis parallels that of RPCholesky~\cite{rpchol} for PSD matrices and relies on the expected one-step Schur complement. We highlight similarities and key differences by recalling certain elements of the analysis of RPCholesky. 

\subsection{Geometric convergence}

Let $\mat{C}^{(1)}$ be the residual after 1 step of RPCholesky applied to $\mat{C}^{(0)} = \mat{C} \succeq 0$, i.e. the result of $1$ step of~\cref{eq:residual_update} with pivot $(i,j)$ sampled with probability $p_{i,j} = \delta_{ij}c_{ii} /\tr \mat{C}$ where $\delta_{ij}$ is the Kronecker delta. Its expected value is given by
\begin{align}\label{eq:rpchol_schur_complement}
\EX[\mat{C}^{(1)} \mid \mat{C}] \;=\; \left( \mat{C} - \frac{\mat{C}^2}{\tr \mat{C}} \right) \;=:\; \Phi(\mat{C}).
\end{align}
The analysis of RPCholesky hinges on this identity; the properties of the map $\Phi$ are heavily used to establish~\cref{theorem:rpchol_bound}.

Let us now consider  $\mat{A}^{(1)}$, the Schur complement after $1$ step of RPLU, with $\mat{A}^{(0)} = \mat{A} \in \mathbb{C}^{n\times m}$, i.e. the result of $1$ step of~\cref{eq:residual_update} with pivots chosen with probability $p_{i,j} = |a_{i,j}|^2/\|\mat{A}\|_F^2$. Its expected value is given by
\begin{align*}
\EX[\mat{A}^{(1)} \mid \mat{A}] \;= \sum_{i,j} p_{i,j} \left( \mat{A} - \frac{1}{a_{i,j}} \mat{l}_j \mat{u}_i^{\!T} \right) = \; \mat{A} - \frac{\mat{A}\mat{A}^*\mat{A}}{\tr(\mat{A}^*\mat{A})},
\end{align*}
where $\mat{l}_j$ denotes column $j$ of $\mat{A}$, and $\mat{u}_i^{\!T}$ is row $i$ of $\mat{A}$.
Despite the nice analogy, this identity is less fruitful in the nonsymmetric setting. For PSD matrices,
the nuclear norm coincides with the trace, and the linearity of the trace allows expectations to commute with the norm. No analogous norm exists for general rectangular or nonsymmetric matrices.

To bypass this difficulty, we study the Gram matrix of the one-step residual, $ {\mat{A}^{(1)}}^* \mat{A}^{(1)}$. A short calculation yields a PSD-order upper bound:
\begin{equation}\label{eq:G1_bound}\begin{aligned}
\EX[{\mat{A}^{(1)}}^* \mat{A}^{(1)} \mid \mat{A}]
&= \sum_{i,j} p_{i,j}
\Big(\mat{A} - \tfrac{1}{a_{i,j}}\,\mat{l}_j \mat{u}_i^{\!T}\Big)^{\!*}
\Big(\mat{A} - \tfrac{1}{a_{i,j}}\,\mat{l}_j \mat{u}_i^{\!T}\Big) \\
&= \mat{A}^*\mat{A} \;-\; \frac{2}{\|\mat{A}\|_F^2}\,\mat{A}^*\mat{A}\,\mat{A}^*\mat{A}
\;+\; \frac{1}{\|\mat{A}\|_F^2}\!\sum_{\substack{i,j\\ a_{i,j}\neq 0}} \overline{\mat{u}_i}\,\|\mat{l}_j\|_F^2\,\mat{u}_i^{\!T} \\
&\preceq 2\,\mat{A}^*\mat{A} \;-\; \frac{2}{\|\mat{A}\|_F^2}\,\mat{A}^*\mat{A}\,\mat{A}^*\mat{A}
\;=\; 2\,\Phi(\mat{A}^*\mat{A}),
\end{aligned}\end{equation}
where we used the fact that $\sum_{\substack{i,j\,:\, a_{i,j} = 0}} \overline{\mat{u}_i} \|\mat{l}_j\|_F^2 \mat{u}_i^T \succeq 0$. Remarkably, the same polynomial map ${\Phi}$ that governs RPCholesky reappears here with two key differences: we obtain an upper bound rather than an equality (though~\Cref{eq:G1_bound} is tight if $a_{ij} \neq 0$ for all $i,j$), and an extra factor of $2$ appears. We use this one-step Gram-matrix bound as a key ingredient to prove the following theorem:
\begin{theorem}\label{theorem:doubling}
Let $\mat{A}\in\mathbb{C}^{n\times m}$.  
After $k$ steps of randomly pivoted LU, the expected squared Frobenius norm of the $k$th Schur complement satisfies
\begin{equation}\label{eq:doubling}
\EX\bigl\|\mat{A} - \mat{\hat{A}}^{(k)}\bigr\|_F^{2}\;\le\;4^{k}\,\bigl\|\mat{A}-\lowrank{\mat{A}}_{k}\bigr\|_F^{2}.
\end{equation}
\end{theorem}

The proof is deferred to \cref{sec:doubling_proof}. 
In particular, \eqref{eq:doubling} shows that RPLU returns an accurate low-rank approximation when the singular values of $\mat{A}$ decay sufficiently fast. For example, if $\sigma_j \le C\,\rho^{\,j-1}$ for $j\ge 3$, with some $C>0$, then
\[
4^{k}\,\bigl\|\mat{A}-\lowrank{\mat{A}}_{k}\bigr\|_F^{2}
=
4^{k}\sum_{j=k+1}^{n}\sigma_{j}^{2}
\;\le\;
\frac{C^2}{1-\rho^{2}}\,(2\rho)^{2k} \longrightarrow 0 \quad \text{if } \rho < \frac{1}{2}.
\]

This bound compares favorably with the worst-case behavior of CPLU (i.e., greedy deterministic pivoting), which requires a geometric decay rate \(\rho < \tfrac14\) to obtain a guarantee of comparable quality~\cite{cortinovis2020maximum}. We conjecture that the factor \(4^{k}\) in~\cref{eq:doubling} is not sharp and can be improved to \(2^{k}\), although we do not expect the exponential dependence on \(k\) to disappear entirely (in contrast with the PSD setting, where it is conjectured that such exponential growth can be removed~\cite[Conjecture 11.2]{epperly2025make}). To see why some exponential growth may be unavoidable, consider a dense unitary matrix \(\mat{A}\in\mathbb{C}^{n\times n}\) with \(a_{ij} \neq 0\) for all \(i,j\). In this case,
\[
\EX\!\left[\left\|\mat{A}^{(1)}\right\|_F^2\right]
\;=\;
2\,\tr\!\bigl(\Phi(\mat{A}^*\mat{A})\bigr)
\;=\;
2\,(n-1)
\;=\;
2\,\bigl\|\mat{A}-\lowrank{\mat{A}}_{1}\bigr\|_F^{2},
\]
where we used that~\cref{eq:G1_bound} holds with equality when all entries of \(\mat{A}\) are nonzero. In numerical experiments, such growth is typically confined to the first few iterations for matrices that are (nearly) unitary and does not persist for larger \(k\).   
Nonetheless, milder growth can occur in practice (unlike the PSD setting); see, e.g.,~\Cref{fig:graphs_result}.

\subsection{Comparison in the PSD case} \label{sec:psd_comparison}

When the initial matrix $\mat{C} \in \mathbb{C}^{n\times n}$ is PSD, RPCholesky is preferable to RPLU, since sampling a diagonal pivot requires only \(\mathcal{O}(n)\) work. Nevertheless, one might ask whether off-diagonal sampling can be competitive in the PSD setting. To investigate this, we consider a PSD-preserving variant of RPLU based on symmetric rank-2 updates (SRPLU).

At iteration \(k\), SRPLU samples a pivot pair \((i,j)\) with probability proportional to \(|\tilde{c}^{(k)}_{ij}|^2\) and applies the rank-2 update
\begin{equation}\label{eq:SRPLU}
  \tilde{\mat{C}}^{(0)} = \mat{C}, \quad
\tilde{\mat{C}}^{(k+1)}
\;=\;
\tilde{\mat{C}}^{(k)}
\;-\;
\mat{L}_{ij}\,\mat{B}_{ij}^{\dagger}\,\mat{L}_{ij}^{\!*},
\qquad
\mat{B}_{ij}
=
\begin{bmatrix}
\tilde{c}_{ii}^{(k)} & \tilde{c}_{ij}^{(k)}\\[2pt]
\tilde{c}_{ji}^{(k)} & \tilde{c}_{jj}^{(k)}
\end{bmatrix},
\quad
\mat{L}_{ij}
=
\begin{bmatrix}
\tilde{\mat{C}}^{(k)}_{:,i} & \tilde{\mat{C}}^{(k)}_{:,j}
\end{bmatrix},
\end{equation}
which preserves positive semidefiniteness: \(\tilde{\mat{C}}^{(k)} \succeq 0\) for all \(k\).\footnote{A potential advantage of this symmetric update is that it can also be applied to symmetric indefinite matrices to obtain a symmetric low-rank decomposition, although we were unable to prove any convergence guarantees in that case.}

The next proposition gives a one-step comparison: in expectation, a single SRPLU rank-2 update reduces the nuclear norm at least as much as a single rank-1 RPCholesky update. This is a local statement and does not imply that the rank-2 scheme dominates the rank-1 scheme after \(k\) steps. It does, however, suggest that off-diagonal sampling can be competitive with diagonal sampling. A related off-diagonal sampling strategy is analyzed in~\cite{steinerberger2024randomly}.

\begin{proposition}\label{prop:SRPLU_vs_RPCholesky}
Let \(\mat{C} \succeq 0\), let $\tilde{\mat{C}}^{(1)}$ be the residual after one SRPLU step, and let ${\mat{C}}^{(1)}$ be the residual after one RPCholesky step. Then
\begin{equation} \label{eq:off-diag-2-diag-1}
\tr\bigl( \EX[ \tilde{\mat{C}}^{(1)} ] \bigr)
\;\leq\;
\tr\bigl( \EX[ {\mat{C}}^{(1)} ] \bigr).
\end{equation}
\end{proposition}
The proof is given in \Cref{sec:SRPLU_vs_RPCholesky_proof}.
In the numerical experiments reported in~\Cref{fig:symmetric_pivoting}, we observe that all three randomized schemes (RPLU, SRPLU, and RPCholesky) behave similarly and substantially outperform greedy LU in the presence of outliers (for example, for the spiral point distribution).

\begin{figure}[h]
  \centering
  \includegraphics[width=1\textwidth]{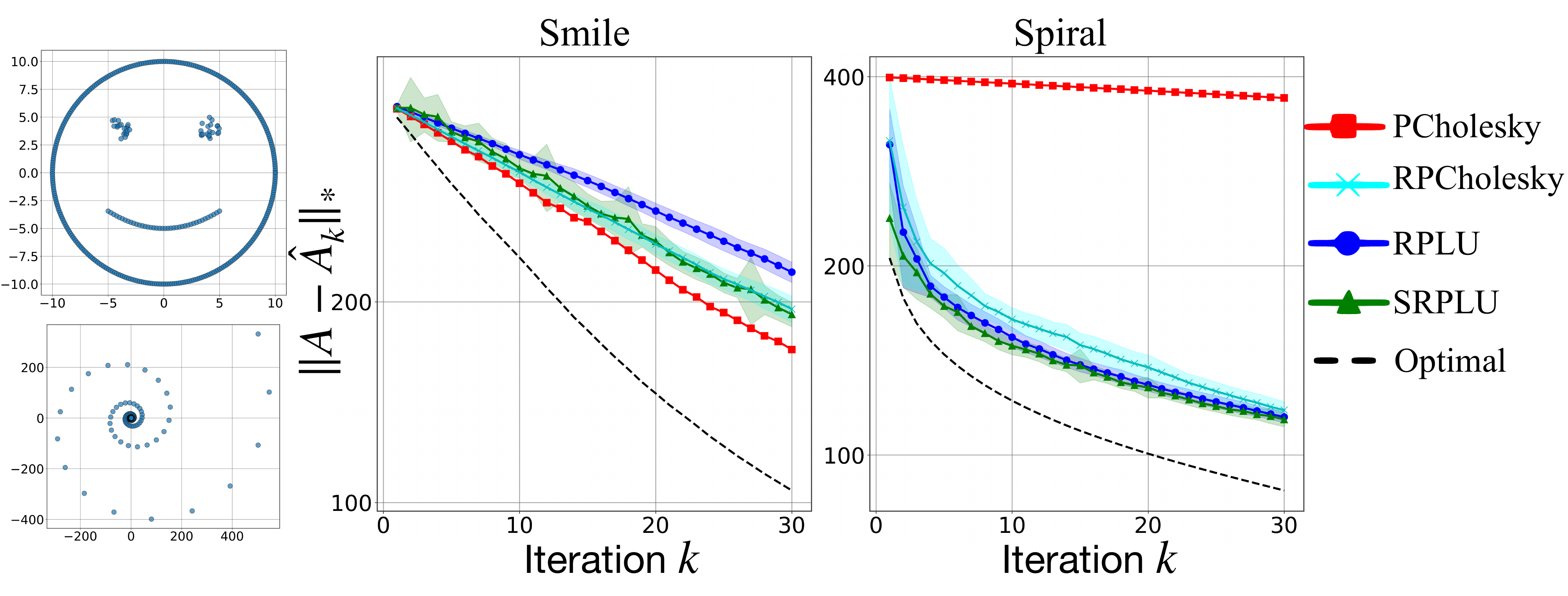}
  \caption{Comparison of low-rank approximation performance on kernel matrices \hbox{$\mat{A}_{ij} = \exp(-0.5\|x_i-x_j\|^2)$} for two point sets shown on left. The nuclear-norm error is averaged over 30 runs, with standard deviation indicated by shaded regions.}
  \label{fig:symmetric_pivoting}
\end{figure}

\subsection{Dynamics of the one-step expected residual} \label{sec:one_step_bound}

The main theorem for RPCholesky (\cref{theorem:rpchol_bound}), which relates the iterate after $k \ge r$ steps to the best rank-$r$ approximation, is driven by the map
\[
\Phi(\mat{C}) \;=\; \EX[\mat{C}^{(1)} \mid \mat{C}],
\]
which gives the expected residual after one step. It is therefore natural to ask whether an analogous argument yields a comparable bound for RPLU; we discuss here why it does not.  

\Cref{theorem:rpchol_bound} for RPCholesky is based on the observation that $\Phi$ is concave with respect to the PSD order~\cite[Lemma 5.3]{rpchol}, so that
\[
\EX[\mat{C}^{(k)} \mid \mat{C}]
\;=\;
\EX[ \Phi(\mat{C}^{(k-1)}) \mid \mat{C}]
\;\preceq\;
\Phi\bigl(\EX[\mat{C}^{(k-1)} \mid \mat{C}]\bigr),
\]
where the equality follows from the law of total expectation and the inequality from the matrix version of Jensen’s inequality. Iterating this yields
\begin{equation}\label{eq:jensen_chain}
  \EX\!\bigl[\mat{C }^{(k)} \mid \mat{C}\bigr]
  \;\preceq\;
  \Phi\bigl(\EX[\mat{C}^{(k-1)} \mid \mat{C}]\bigr)
  \;\preceq\;
  \cdots
  \;\preceq\;
  \Phi^{\circ k}(\mat{C}).
\end{equation}
A separate dynamical system argument is then used to bound $\Phi^{\circ k}(\mat{C})$, which leads to the estimate in~\cref{theorem:rpchol_bound}.

An analogous argument for RPLU can be made by working with the Gram matrix $\mat{A}^{(k)*}\mat{A}^{(k)}$. Using the one-step bound in~\cref{eq:G1_bound}, we obtain
\begin{equation}\label{eq:jensen_chain_rplu}
  \EX\!\bigl[\mat{A}^{(k) *}\mat{A}^{(k)} \mid \mat{A}\bigr]
\;\preceq\;
2\,\EX\!\bigl[\Phi(\mat{A}^{(k-1) *}\mat{A}^{(k-1)}) \mid \mat{A}\bigr]
  \;\preceq\;
  \cdots
  \;\preceq\;
  2^{k}\,\Phi^{\circ k}(\mat{A}^* \mat{A}).
\end{equation}
This suggests geometric decay only if $\Phi^{\circ k}(\mat{A}^* \mat{A})$ decays geometrically faster than $2^{k}$. The next result shows that this is essentially never the case, regardless of the singular value decay of $\mat{A}$. 

\begin{proposition}\label{prop:trace_ratio}
For any $\mat{C} \succeq 0$ of rank $r \geq 2$,
\begin{align}
\lim_{k\to\infty}
\frac{\tr\,\Phi^{\circ(k+1)}(\mat{C})}{\tr\,\Phi^{\circ k}(\mat{C})}
\;=\; 1-\frac{1}{r}.\label{eq:trace_ratio_2}
\end{align}

\end{proposition}

The proof appears in~\Cref{sec:dynamics_proofs}.
\Cref{prop:trace_ratio} highlights two issues.
First, \cref{eq:trace_ratio_2} implies that the RPLU upper bound in~\cref{eq:jensen_chain_rplu} grows without bound for any matrix of rank $r \geq 3$.
Second, bounds obtained from the dynamics of $\Phi^{\circ k}$, including~\cref{theorem:rpchol_bound} for RPCholesky, fail to capture fast geometric convergence for large $k$, regardless of the singular value decay rate.
This is in contrast with~\cref{lem:rchol_doubling} for RPCholesky and~\cref{theorem:doubling} for RPLU, which both guarantee arbitrarily fast decay for matrices with sufficiently rapid geometric singular value decay.

This behavior is illustrated in~\Cref{fig:rpchol_bounds}, where the bound in~\cref{theorem:rpchol_bound} is subgeometric even when the singular values, and observed RPCholesky residual decays geometrically.
On the other hand, both~\cref{theorem:doubling} for RPLU and~\cref{lem:rchol_doubling} for RPCholesky guarantee geometric decay whenever the singular values decay faster than $(1/2)^k$, but supply no guarantees when $\rho \geq \frac{1}{2}$. These observations suggest that sharper bounds for both RPCholesky and RPLU may require tracking multi-step dynamics of the residual.

\begin{figure}[htbp]
\centering
\includegraphics[width=\textwidth]{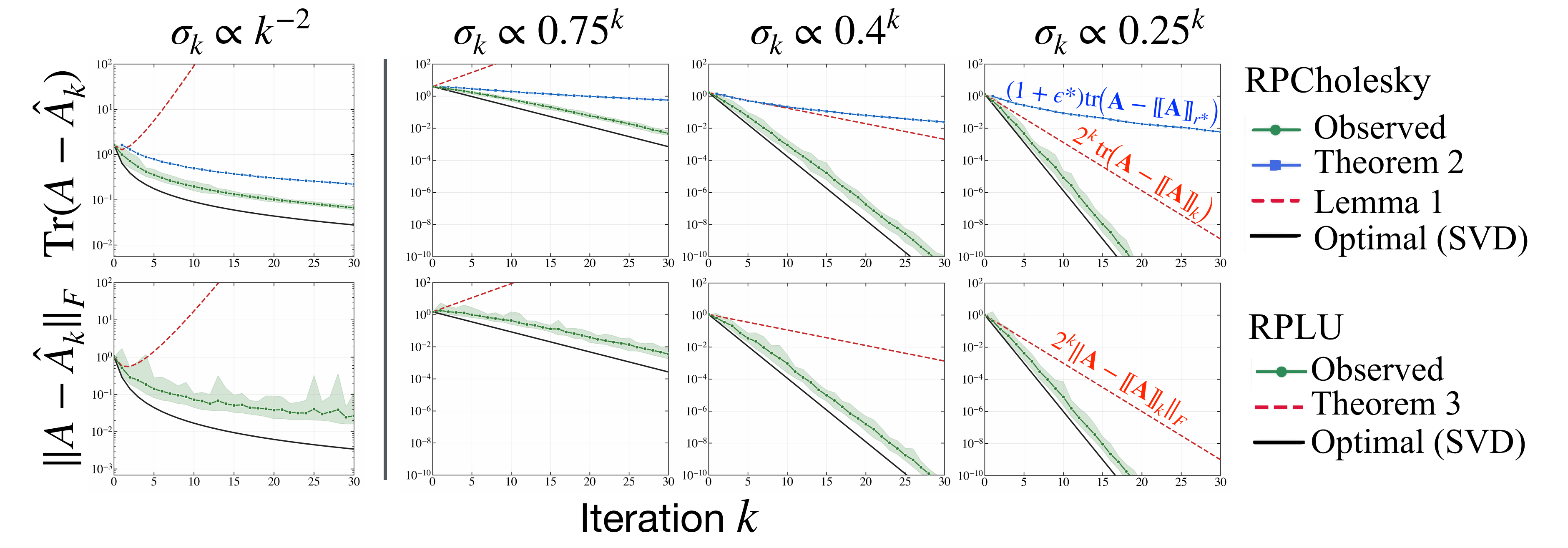}
\caption{Illustration of the bounds in~\cref{theorem:rpchol_bound} and~\cref{lem:rchol_doubling} for RPCholesky and~\cref{theorem:doubling} for RPLU on random $200 \times 200$ matrices with different singular value decay rates.
Top row: RPCholesky with polynomial decay $\sigma_k \propto 1/k^2$ and geometric decay $\sigma_k = \rho^k$ for $\rho \in \{0.75, 0.4, 0.25\}$.
Bottom row: RPLU on the same decay profiles.
To compute the bound in~\cref{theorem:rpchol_bound}, for each $k$ we perform a grid search over $\varepsilon$ and $r$ to find the values that minimize the bound.
When the singular values decay geometrically with rate $\rho < \frac{1}{2}$,~\cref{lem:rchol_doubling} and~\cref{theorem:doubling} provide a loose bound with geometric decay for RPCholesky and RPLU, respectively, and no guarantees when $\rho \geq \frac{1}{2}$.
In contrast, the RPCholesky bound in~\cref{theorem:rpchol_bound} provides a sub-geometric bound in all cases.
In all experiments with geometric decay, the observed residuals (mean over 30 trials with min/max bands) track closely with the optimal rank-$k$ approximation error, significantly outperforming all theoretical bounds.
}
\label{fig:rpchol_bounds}
\end{figure}

\section{ Applications to large structured matrices } \label{sec:practical_algorithm}

We now turn to practical situations where RPLU is a competitive algorithm for low-rank computation.
A na\"ive implementation of~\Cref{alg:rplu} is expensive: each iteration needs to sample a pivot and form the Schur complement,
each of which costs $\mathcal{O}(nm)$ per iteration. 
In this setting, alternative methods such as randomized SVD~\cite{halko2011finding} or leverage‐score/CUR methods~\cite{mahoney2009cur} should be preferred, as they are both faster and provide stronger theoretical guarantees.

However, if the row norms of the initial matrix $\mat{A}$ can be computed efficiently and fast matrix--vector products with $\mat{A}$ and $\mat{A}^T$ are available, then RPLU can be implemented much more efficiently and with low memory overhead, making it attractive for very large problems on memory-constrained hardware such as GPUs.

\subsection{Efficient, low-memory implementation}

First, suppose we know the row norms of the residual $\mat{A}^{(k)}$. Then we can sample a pivot in two stages: sample a row $i$ with probability proportional to its squared norm, $\prob(i\mid\mat{A}^{(k)}) \propto \|\mat{A}^{(k)}_{i,:}\|_2^2$, and then sample a column $j$ with probability proportional to the squared entry in that row, $\prob(j\mid i,\mat{A}^{(k)}) \propto |a^{(k)}_{i,j}|^2$. The resulting joint probability is
\[
\prob(i,j \mid\mat{A}^{(k)}) = \frac{\|\mat{A}^{(k)}_{i,:}\|_2^2}{\|\mat{A}^{(k)}\|_F^2} \cdot \frac{|a^{(k)}_{i,j}|^2}{\|\mat{A}^{(k)}_{i,:}\|_2^2} = \frac{|a^{(k)}_{i,j}|^2}{\|\mat{A}^{(k)}\|_F^2},
\]
which matches the desired distribution.

We now show how to compute these row norms without forming $\mat{A}^{(k)}$ explicitly, assuming we have access to the initial row norms, and a fast matrix--vector products with $\mat{A}$ and $\mat{A}^T$ are available.
The key idea is to maintain the low-rank approximation in CUR form rather than in LU form, similar to
the sparse version of CPQR described in~\cite{stewart1999four} and the low-memory implementation of pivoted Cholesky proposed in~\cite{epperly2024embrace}. That is, we write the rank-$k$ approximation as iteration $k$ as
\[
\hat{\mat{A}}_k := \mat{C}_k\,\mat{U}_k\,\mat{R}_k,
\]
where $\mat{C}_k = \mat{A}_{:,\J_k} \in \mathbb{C}^{n\times k}$, $\mat{R}_k = \mat{A}_{\I_k,:} \in \mathbb{C}^{k\times m}$,  $\I_k \subset [n]$ and $\J_k \subset [m]$ are the sets of selected rows and columns respectively,
and $\mat{U}_k = \mat{W}_k^{-1}$ with $\mat{W}_k = \mat{A}_{\I_k,\J_k} \in \mathbb{C}^{k\times k}$. 
By writing the approximation in this form, operations with the residual $\mat{A}^{(k)} = \mat{A} - \hat{\mat{A}}_k $ at step $k$ can be performed efficiently with matvecs with $\mat{A}$ (or $\mat{A}^T$). For example, column $i$ of $\mat{A}^{(k)}$ can be computed efficiently with two matvecs:
\[
\mat{A}^{(k)}_{:,i} = \mat{A}^{(k)}e_i = (\mat{A} - \mat{\hat{A}}_k)e_i = (\mat{I} - \mat{A}\mat{S}_{\J_k}^T\,\mat{U}_k\,\mat{S}_{\I_k})\,\mat{A} e_i
\]
where $e_i$ is the $i$-th standard basis vector, $\mat{I}$ is the identity matrix, and $\mat{S}_{\mathcal{H}} := \mat{I}_{:,\mathcal{H}}$ is the selection matrix for a subset of indices $\mathcal{H}$.
By storing the approximation in this way, the only quantities necessary to perform RPLU are: $\I_k$, $\J_k$, and $\mat{U}_k$, and the row norms of $\mat{A}^{(k)}$.
We now detail how to update the row norms and the block inverse.

\subsubsection*{Core matrix update}
At each step, a new row index $i_{k+1}$ and column index $j_{k+1}$ is selected and added to $\I_k$ and $\J_k$ respectively. To update the core matrix, we first extract the corresponding vectors from $\mat{A}$ and $\mat{A}^T$:
\begin{align*}  
r^\top &:= \mat{A}_{i_{k+1},:}, \quad
z^\top := \mat{A}_{i_{k+1},\J_k}, \\
c &:= \mat{A}_{:,j_{k+1}}, \quad
w := \mat{A}_{\I_k,j_{k+1}}, \\
\omega &:= \mat{A}_{i_{k+1},j_{k+1}}.
\end{align*}
The updated core matrix is
\[
\mat{W}_{k+1}
=
\mat{A}_{\I_{k+1},\J_{k+1}}
=
\begin{bmatrix}
\mat{W}_k & w \\
z^\top & \omega
\end{bmatrix},
\qquad
\sigma \coloneqq \omega - z^\top \mat{U}_k w.
\]
Using block inversion, $\mat{U}_{k+1} = \mat{W}_{k+1}^{-1}$ may be updated as
\begin{equation} \label{eq:block-inverse-update}
  \mat{U}_{k+1}
  =
  \begin{bmatrix}
  \mat{U}_k + \sigma^{-1}\mat{U}_k w\, z^\top \mat{U}_k & - \sigma^{-1}\mat{U}_k w \\
  -\,\sigma^{-1} z^\top \mat{U}_k & \sigma^{-1}
  \end{bmatrix}.
\end{equation}
Thus, $\mat{U}_{k+1}$ can be updated from $\mat{U}_{k}$ with $\mathcal{O}(k^2)$ work.

If the singular values of $\mat{A}$ decay rapidly, $\mat{W}_{k}$ will typically be ill-conditioned and this update may be numerically unstable. Alternatively, it may be preferable to store $\mat{W}_{k}$ and compute the required solves $\mat{W}_{k}^{-1}v$ using a more stable method, such as a QR decomposition of $\mat{W}_{k}$; moreover, this QR decomposition can be updated from $\mat{W}_{k}$ to $\mat{W}_{k+1}$ in $\mathcal{O}(k^2)$ work~\cite{daniel1976reorthogonalization}.\footnote{Our implementation uses the QR approach by default but we recompute the QR factorization rather than using the low-rank update, since its cost is negligible in our experiments.}

\subsubsection*{Row norm updates}

The only ingredient left to compute are the row norms of $\mat{A}^{(k+1)}$.
Let $r^{(k)} := \mat{A}^{(k)}_{i_{k+1},:}$ and $c^{(k)} := \mat{A}^{(k)}_{:,j_{k+1}}$ be row and column of the residual matrix $\mat{A}^{(k)}$ chosen at step $k+1$.
Using the rank-$1$ update of the residual,
$
\mat{A}^{(k+1)}
= \mat{A}^{(k)} - \,c^{(k)}r^{(k)}  \big/ r^{(k)}_{j_{k+1}} ,
$
the squared row norms $n^{(k+1)}_{\mathrm{row}}=\operatorname{diag}(\mat{A}^{(k+1)}\mat{A}^{(k+1)*})$ may be computed as:
\begin{align*}
n^{(k+1)}_{\mathrm{row}}
&= n^{(k)}_{\mathrm{row}}
- 2\,\Re\!\left(
{(\mat{A}^{(k)}\overline{r^{(k)}}^T)\odot \overline{c^{(k)}} \big/ \overline{r^{(k)}_{j_{k+1}}}}
\right)
+ \frac{\|r^{(k)}\|_2^2}{|r^{(k)}_{j_{k+1}}|^2}\,|c^{(k)}|^{\odot 2}.
\end{align*}
Here $\odot$ denotes the element-wise product and $|x|^{\odot 2} := x \odot \overline{x}$.
The dominant cost of this update is to compute $\mat{A}^{(k)}\overline{r^{(k)}}^{T} = (\mat{A}-\hat{\mat{A}}_k)\overline{r^{(k)}}^T$, which requires two matrix-vector products with $\mat{A}$ plus $\mathcal{O}(k^2+n)$ vector operations.

\subsubsection*{Complexity analysis}
In addition to the initial row-norm computation, the total complexity of $k$ iterations of RPLU in CUR form (\Cref{alg:rplu-cur}) is
$\mathcal{O}(k(n+m)+k^3) + 4k\mathcal{M}(\mat{A}) + 2k\mathcal{M}(\mat{A}^T)$ operations and requires $\mathcal{O}(k^2+n+m)$ memory where 
 $\mathcal{M}(\mat{A})$ denotes the cost of a matrix-vector product (matvec) with $\mat{A}$.\footnote{For simplicity, we count each application of $\mat{A}$ or $\mat{A}^T$ as a full matvec. However, only one of these requires the full matrix, while the remaining operations can be carried out using only the rows/columns selected so far. Depending on the structure of the matrix, these operations may be implemented more efficiently than a full matvec.}

For comparison, a randomized SVD can also exploit fast matvecs and requires
$\mathcal{O}(k^2(n+m)+k^3) + k\mathcal{M}(\mat{A}) + k\mathcal{M}(\mat{A}^T)$ work and $\mathcal{O}(k(n+m)+k^2)$ memory.
Thus, for large $n$ and $m$, RPLU reduces both storage and arithmetic by a factor of $\mathcal{O}(k)$ relative to randomized SVD, at the cost of three times more matvecs.
In many settings with moderate $k$, randomized SVD will still be faster and more accurate in practice, since it is non-iterative, can leverage level-3 BLAS operations and has stronger approximation guarantees. The main practical advantage of RPLU is its $\mathcal{O}(k^2+n+m)$ storage, rather than $\mathcal{O}(k^2 + k(n+m))$, which enables high-rank approximations of very large problems on memory-constrained hardware as we show in~\Cref{sec:numerical}.

\begin{algorithm}[H]
  \caption{RPLU in CUR form}
  \label{alg:rplu-cur}
  \begin{algorithmic}
    \State Input: Matrix $\mat{A} \in \mathbb{C}^{n \times m}$, target rank $r$
    \State Initialize $\I_0 \leftarrow \emptyset$, $\J_0 \leftarrow \emptyset$, $U_0 \leftarrow \mathbb{R}^{0 \times 0}$
    \State $n_0 \leftarrow \|\mat{A}[i,:]\|_2^2$ for all $i \in [n]$ \Comment{Initial row norms of $\mat{A}$}
    \For{$k = 0$ to $r-1$}
    \State \textbf{Step 1: Sample pivot row and column}
    \State Sample a row index $i_{k+1} \propto n_{k}$ 
    \State $r^\top \leftarrow \mat{A}[i_{k+1},:]$ \Comment{Row of $\mat{A}$ -- $1$ $\mathcal{M}(\mat{A}^T)$}
    \State $\hat{r} \leftarrow r - \mat{A}[\I_k,:]^{\top}  \, \mat{U}_{k}^{\top} \, r[\J_k]$ \Comment{Row of $\mat{A}^{(k)}$ -- $1$ $\mathcal{M}(\mat{A}^T)$}
    \State Sample a column index $j_{k+1} \propto |\hat{r}|^{\odot 2}$ 
    \State $c \leftarrow \mat{A}[:,j_{k+1}]$ \Comment{Column of $\mat{A}$ -- $1$ $\mathcal{M}(\mat{A})$}
    \State $\hat{c} \leftarrow c - \mat{A}[:,\J_k] \mat{U}_{k}  c[\I_k] $ \Comment{Column of $\mat{A}^{(k)}$ -- 1 $\mathcal{M}(\mat{A})$}

    \State \textbf{Step 2: Update row norms}
    \State $l \leftarrow \hat{r} / \hat{r}[j_{k+1}]$
    \State $g \leftarrow \mat{A} \, l$  \Comment{1 $\mathcal{M}(\mat{A})$ }
    \State $v \leftarrow \mat{A}[:,\J_k] \mat{U}_{k} g[\I_k] $  \Comment{1 $\mathcal{M}(\mat{A})$ }
    \State $n_{k+1} \leftarrow n_{k} - 2 \, \Re\big((g -v) \odot \overline{\hat{c}}\big) + \|l\|_2^2 \, |\hat{c}|^{\odot 2}$ 

    \State \textbf{Step 3: Update CUR}
    \State Update block inverse $\mat{U}_{k+1}$ from $\mat{U}_{k}$ with \cref{eq:block-inverse-update}
    \State $\I_{k+1} \leftarrow \I_{k} \cup \{i_{k+1}\}$
    \State $\J_{k+1} \leftarrow \J_{k} \cup \{j_{k+1}\}$

    \EndFor
    \State Output: $\I_r$, $\J_r$, $\mat{U}_r$
  \end{algorithmic}
\end{algorithm}

\subsection{Numerical results} \label{sec:numerical}

We benchmark the practical performance of~\Cref{alg:rplu-cur} on very large structured matrices using a GPU, where computation is fast but memory is constrained. 
We compare against the following alternative algorithms:
\begin{itemize}
  \item \textbf{CUR via C2PLU.} Implemented identically to~\Cref{alg:rplu-cur}, but with greedy pivot selection: it replaces the two sampling steps with max-selection. Its cost is the same as RPLU.

  \item \textbf{CUR via sparse $Q$-less pivoted QR}~\cite{stewart1999four}.
  We include this algorithm because its memory complexity matches that of~\Cref{alg:rplu-cur}. 
  The algorithm (detailed in~\Cref{alg:qr_cur}) computes two $Q$-less pivoted QR factorizations, $\mat{Q}_1 \mat{R}_1 = \mat{A}_{:,\J}$ and $\mat{Q}_2 \mat{R}_2 = (\mat{A}_{\I,:})^T$, and then forms the projective core matrix $\mat{U} = \mat{A}_{:,\J}^{\dagger} \mat{A}\, \mat{A}_{\I,:}^{\dagger}$ via 
  \begin{equation} \label{eq:U_cpqr}
    \mat{U} = (\mat{R}_1^T\mat{R}_1)^{-1} \mat{A}_{:,\J}^T \mat{A}\, \mat{A}_{\I,:}^T (\mat{R}_2^T\mat{R}_2)^{-1}.
  \end{equation}
  Its memory complexity is $\mathcal{O}(k^2 + m + n)$ and its time complexity is $\mathcal{O}(k(n+m) + k^2 + 9k\, \mathcal{M}(A) + 6k\, \mathcal{M}(A^T))$. We test both greedy column pivoting (CPQR)~\cite{stewart1999four} and random column pivoting (RPQR)~\cite{rpchol}.
  \item \textbf{Randomized SVD.} Its memory complexity is $\mathcal{O}(k(n+m))$ and time complexity $\mathcal{O}((n+m)k^2 + k\,\mathcal{M}(A) + k\,\mathcal{M}(A^T))$. This method is infeasible for very large matrices when memory is constrained, but we include it as a representative of methods with $\mathcal{O}(k(n+m))$ memory since it often achieves the best speed and accuracy in that class.
\end{itemize}
All algorithms are implemented in JAX~\cite{jax2018github} and run on an NVIDIA A100 GPU with 80\,GB of memory.

We first test these algorithms on a square multi-level Toeplitz matrix arising from the discretization of a non-symmetric drifted Gaussian kernel:
\begin{equation}  \label{eq:kernel}
K(x,y,z) =  \exp\!\left(-\frac{1}{2}\left(\frac{(x-\delta)^2}{\sigma^2} + \frac{y^2}{\sigma^2} + \frac{z^2}{\sigma^2}\right)\right),
\end{equation}
with shift $\delta = 50$ and standard deviation $\sigma = 80$. We discretize the domain $[0, 320]^3$ on a Cartesian grid with $n = N^3$ points for $N \in \{128, 320\}$ (step sizes $h = 320/N$). This matrix exhibits geometrically decaying singular values and admits fast FFT-based matvecs: $\mathcal{M}(A) = \mathcal{M}(A^T) = \mathcal{O}(n \log n)$.
\Cref{fig:cur_vs_svd} summarizes timing and approximation error; the main takeaways are as follows:
\begin{itemize}
\item \textbf{Randomized SVD} is preferable for the coarser discretization ($N=128$) at ranks $k < 512$: it is roughly $2\times$ faster than RPLU and achieves about an order-of-magnitude smaller error before running out of memory (OOM) at $k=512$. For the finer discretization ($N=320$), however, it becomes memory-limited almost immediately (OOM at $k=16$).
\item \textbf{QR-based methods} are accurate at low ranks, matching the randomized SVD, but are roughly $2\times$ slower than the LU-based methods. Their low memory footprint allows high-rank approximations, but accuracy stagnates around $k \approx 250$ due to two sources of numerical instability: (i) ill-conditioning of the core matrix $\mat{U}$ in~\cref{eq:U_cpqr}, where $\kappa_2(\mat{A}_{:,\J}^T \mat{A}\, \mat{A}_{\I,:}^T)$ can grow as large as $\kappa_2(\mat{A}_{\I,\J})^3$; and (ii) the column-norm estimates becoming negative around rank $700$.\footnote{This instability is specific to the low-memory, $Q$-less implementation; QR-based CUR methods can be made stable, see e.g.~\cite{epperly2025make,ekentaspectrum}. Using an interpolative rather than projective CUR could also improve stability at the cost of some accuracy; we do not explore this here.}
\item \textbf{C2PLU} has the same runtime, stability, and memory footprint as RPLU, but achieves smaller error in this example. 
\end{itemize}
In summary, RPLU and C2PLU are the only methods able to compute high-rank, accurate approximations for this problem, owing to their small memory footprint and relatively good numerical stability.

\begin{figure}[htbp]
  \centering
  \includegraphics[width=0.8\textwidth]{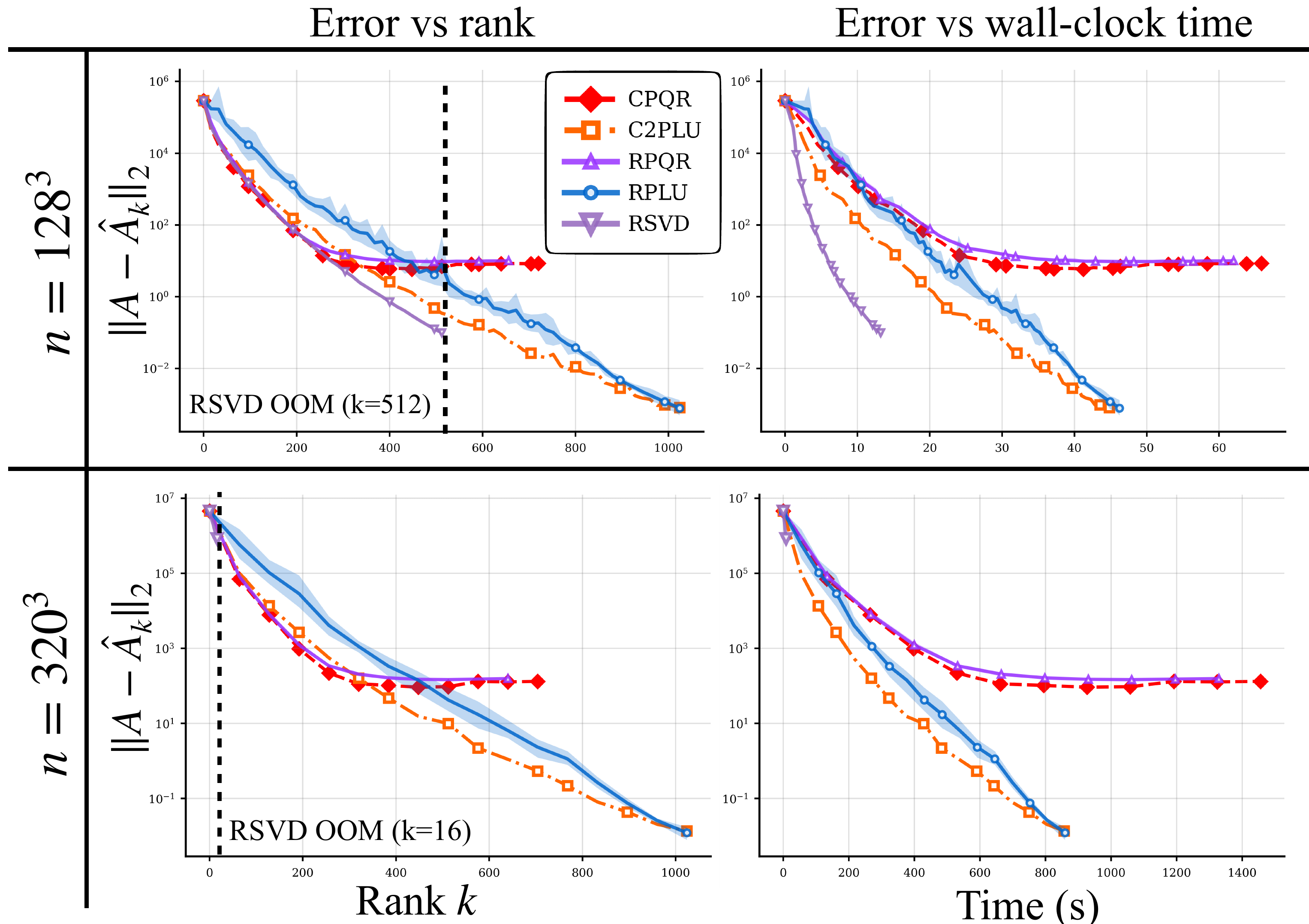}
  \caption{Approximation quality and timing of the various algorithms on a multi-level Toeplitz matrix with $128^3$ and $320^3$ grid points on a GPU. ``OOM'' denotes when the randomized SVD runs out of memory in each case. The $2$-norm is estimated using a Lanczos iteration using matrix-vector products with $\mat{A}$ and $\mat{\hat{A}}$.}
  \label{fig:cur_vs_svd}
\end{figure}

\subsection{Low-memory preconditioning} \label{sec:preconditioning}
To illustrate the usefulness of the low-rank approximation methods with small memory footprint, we consider solving large linear systems on a GPU with an iterative solver.
Specifically, we consider linear systems $\mat{K} x = b$, where $\mat{K} = \mat{B} + \mat{A}$, $ \mathcal{M}(\mat{B}^{-1})$ is cheap and $\mat{A}$ is a term well-approximated by a low-rank factorization $\mat{A} \approx \mat{C}_k \mat{U}_k \mat{R}_k$. 
This type of linear system arises in the numerical solution of differential equations and fitting kernel models.

In that situation, it is natural to approximate $\mat{K}$ as $\mat{P} := \mat{B} + \mat{C}_k \mat{U}_k \mat{R}_k$, and to use $\mat{P}^{-1}$ as a preconditioner.
By the Sherman--Morrison--Woodbury identity~\cite[Section 2.1.4]{golub2013matrix}, the inverse is:
\[
\mat{P}^{-1} = \mat{B}^{-1} - \mat{B}^{-1} \mat{C}_k (\mat{U}_k^{-1} + \mat{R}_k \mat{B}^{-1} \mat{C}_k)^{-1} \mat{R}_k \mat{B}^{-1}.
\]
Given a CUR approximation of $\mat{A}$, the only unknown matrix is $\mat{G}_k := \mat{R}_k \mat{B}^{-1} \mat{C}_k$, which can be formed in $2k\mathcal{M}(\mat{A}) + k\mathcal{M}(\mat{B}^{-1})+ \mathcal{O}(k^3)$ operations.
Then, the cost of a matvec with the preconditioner is $\mathcal{M}(\mat{P}^{-1}) = 2 \mathcal{M}(\mat{A}) + 2 \mathcal{M}(\mat{B}^{-1}) + \mathcal{O}(k^2 + n + m )$ operations and requires only $\mathcal{O}(k^2 + n +m )$ memory.

We demonstrate this on a linear system arising from a large-scale integro-differential boundary value problem on the domain $\Omega = [0, 320]^3$:
\begin{equation}\label{eq:integro_diff}
-\Delta u(\mathbf{x}) + \int_{\Omega} K(\mathbf{x} - \mathbf{y})\, u(\mathbf{y})\, d\mathbf{y} = f(\mathbf{x}), \qquad \mathbf{x} \in \Omega,
\end{equation}
with homogeneous Dirichlet boundary conditions $u|_{\partial\Omega} = 0$, where $K$ is the drifted Gaussian kernel from~\cref{eq:kernel}. Discretizing with finite differences on a grid with $n = 320^3$ points yields $\mat{K} = \mat{B} + \mat{A}$, where $\mat{B}$ is the discrete Laplacian (invertible via the discrete sine transform, so $\mathcal{M}(\mat{B}^{-1}) = \mathcal{O}(n \log n)$) and $\mat{A}$ is the Toeplitz matrix discretizing the integral operator. 

\Cref{fig:preconditioner} compares GMRES(20) convergence for a random right-hand side using different preconditioners: unpreconditioned, and rank-$512$ CUR approximations computed by RPLU, C2PLU, RPQR, and CPQR.
Due to memory constraints (80\,GB), the randomized SVD can only provide a rank-$16$ approximation, which does not improve convergence (not shown).
All CUR-based preconditioners greatly accelerate GMRES, reaching relative error $\sim 10^{-12}$ in about 50 restarts.
The factorization dominates total runtime: approximately 10 minutes for RPLU and C2PLU, and 20 minutes for RPQR and CPQR.

\begin{figure}[htbp]
  \centering
  \includegraphics[width=0.8\textwidth]{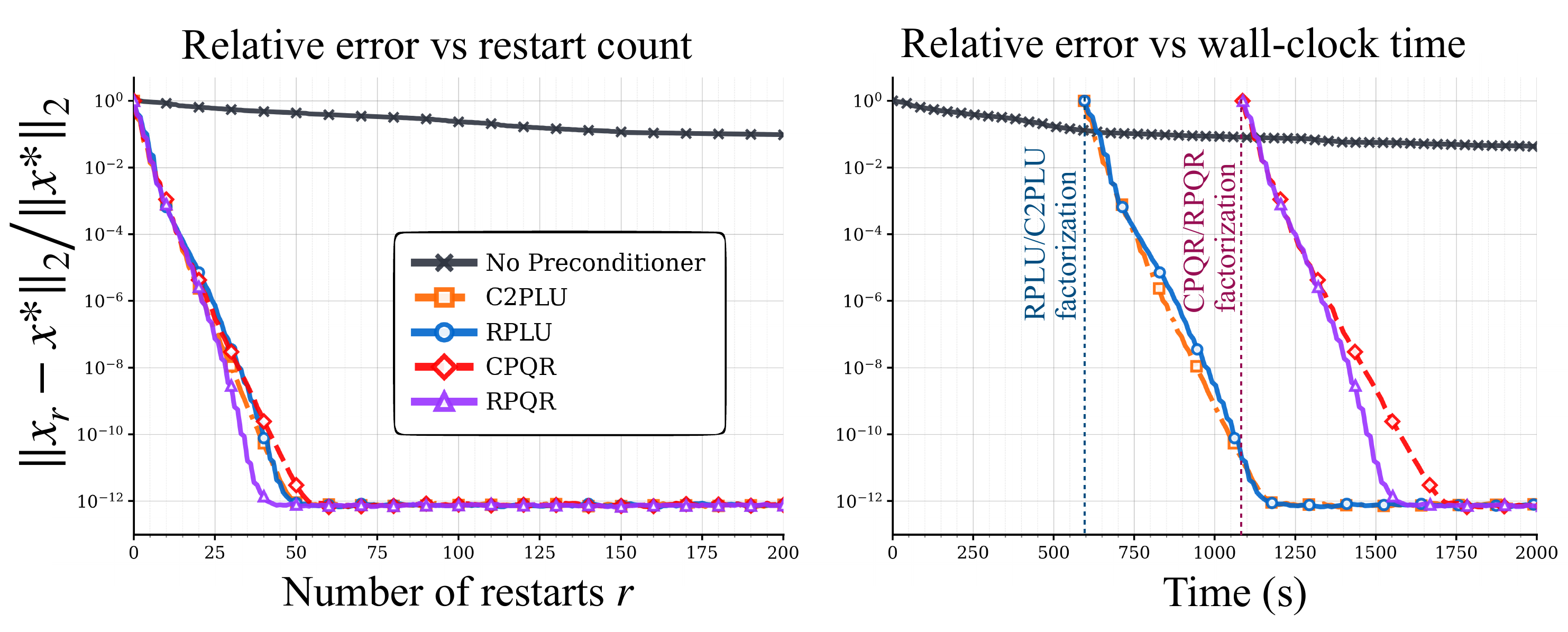}
  \caption{Convergence and timing of three GPU-based GMRES solvers for an integro-differential equation discretized on a $320^3$ grid ($n \approx 3.2 \times 10^7$). We compare unpreconditioned GMRES (black), and GMRES with a rank-$512$ CUR preconditioner computed by RPLU (blue), C2PLU (orange), CPQR (red) and RPQR (purple). }
  \label{fig:preconditioner}
\end{figure}

\subsection{CUR decomposition of sparse matrices}\label{sec:graph_laplacian_lowrank}
CUR decompositions are a natural choice for sparse matrices, since the factors $\mat{C}$ and $\mat{R}$ inherit the sparsity of $\mat{A}$. A common application is to graph adjacency matrices, where the selected rows and columns identify the most ``important" nodes. This idea underlies algorithms for approximate max-cut~\cite{drineas2008sampling}, anomaly detection~\cite{sun2007less}, graph matching~\cite{khan2019image}, and image processing~\cite{henneberger2023hyperspectral}.

We benchmark the four low-memory CUR methods from~\Cref{sec:numerical} on ten large square sparse matrix from the SuiteSparse Matrix Collection~\cite{suite_sparse} detailed in~\Cref{tab:suitesparse-links}, with number of rows ranging from $0.9$ to $36$ millions (\Cref{tab:suitesparse-links}). Results are shown in~\Cref{fig:graphs_result}, sorted by how well each matrix admits a low-rank CUR approximation (i.e., lowest relative Frobenius error achieved).

\begin{table}[h!]
  \centering
  \small
  \begin{tabular}{l l l r r}
  \toprule
  Graph & SuiteSparse ID & Description & $n$ & nnz \\ 
  \midrule
  FullChip & \href{https://sparse.tamu.edu/Freescale/FullChip}{Freescale/FullChip} & Circuit simulation & 2,987,012 & 26,621,990 \\
  circuit5M dc & \href{https://sparse.tamu.edu/Freescale/circuit5M_dc}{Freescale/circuit5M\_dc} & Circuit simulation   & 3,523,317 & 18,583,332 \\
  mawi-A & \href{https://sparse.tamu.edu/MAWI/mawi_201512012345}{MAWI/mawi\_201512012345} & Network traffic & 18,571,154 & 54,174,174 \\
  mawi-B & \href{https://sparse.tamu.edu/MAWI/mawi_201512020000}{MAWI/mawi\_201512020000} & Network traffic & 35,991,342 & 105,048,180 \\
  StocF-1465 & \href{https://sparse.tamu.edu/Janna/StocF-1465}{Janna/StocF-1465} & PDE discretization & 1,465,137 & 21,005,389 \\
  Freescale2 & \href{https://sparse.tamu.edu/Freescale/Freescale2}{Freescale/Freescale2} & Circuit simulation & 2,999,349 & 24,355,389 \\
  emilia & \href{https://sparse.tamu.edu/Janna/Emilia_923}{Janna/Emilia\_923} & Structural mechanics & 923,136 & 40,373,538 \\
  Long Coup \texttt{dt0} & \href{https://sparse.tamu.edu/Janna/Long_Coup_dt0}{Janna/Long\_Coup\_dt0} & PDE discretization & 1,470,152 & 21,101,558 \\
  wiki-Talk & \href{https://sparse.tamu.edu/SNAP/wiki-Talk}{SNAP/wiki-Talk} & Wikipedia talk network & 2,394,385 & 5,021,410 \\
  rajat31 & \href{https://sparse.tamu.edu/Rajat/rajat31}{Rajat/rajat31} & Circuit simulation & 4,690,002 & 20,316,253 \\
  \bottomrule
  \end{tabular}
  \caption{SuiteSparse matrix collection entries used in experiments; sorted by approximability (lowest relative Frobenius error of the CUR decomposition).}
  \label{tab:suitesparse-links}
  \end{table}
The results depend on how well each matrix admits a low-rank approximation:
\begin{enumerate}
  \item For highly approximable matrices (e.g., ``FullChip''), LU-based methods outperform QR-based methods in both speed and accuracy, owing to better numerical stability.
  \item For moderately approximable matrices, QR- and LU-based methods achieve similar accuracy, but LU-based methods are $1.2$--$2.5\times$ faster.
  \item For poorly approximable matrices (e.g., ``wiki-Talk''), the LU-based residual actually increases with the number of steps, whereas the QR-based residual does not.
\end{enumerate}
Notably, the greedy schemes (C2PLU, CPQR) outperform their randomized counterparts in nearly all of these experiments, despite the stronger theoretical guarantees of the randomized methods and the adversarial examples in~\cref{fig:rplu_intro} where greedy pivoting fails.

\begin{figure}[htbp]
  \centering
  \includegraphics[width=\textwidth]{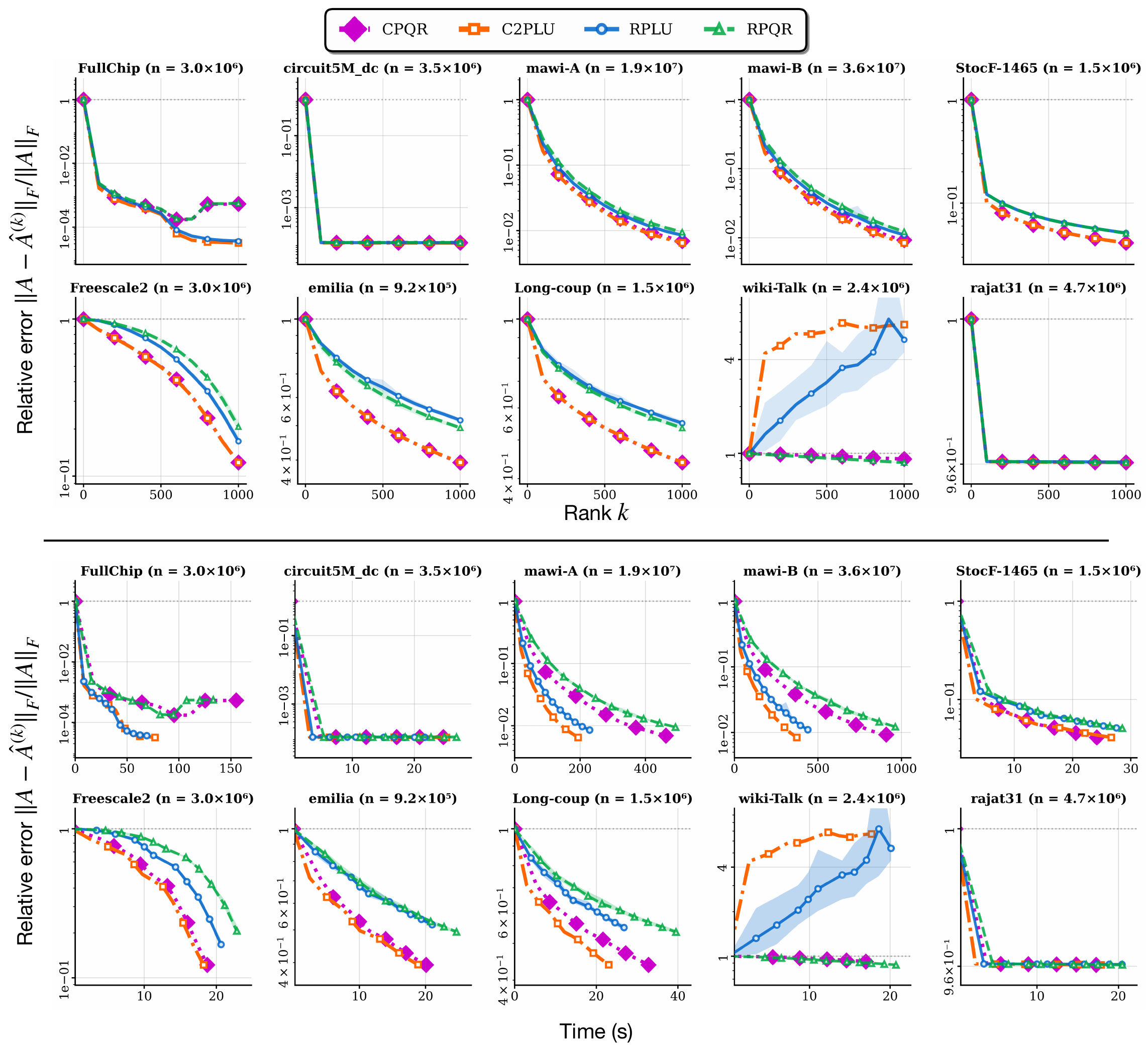}
  \caption{Comparison of the four low-memory CUR methods on ten large sparse matrices from SuiteSparse (\Cref{tab:suitesparse-links}), sorted by approximability. Top row: relative Frobenius error vs.\ rank. Bottom row: relative error vs.\ runtime. LU-based methods are faster and often achieve the same accuracy but can fail for poorly approximable matrices (e.g., ``wiki-Talk'').} 
  \label{fig:graphs_result}
\end{figure}

\section{Applications to Cauchy-like and related matrices} \label{sec:cauchy}
Low-rank approximation methods based on Gaussian elimination are particularly attractive when both a matrix and its Schur complements share an exploitable structure. 
A prime example that gives rise to partial Cholesky approximations is the class of PSD matrices.
Another example, which we explore in this section, are Cauchy-like matrices.
\subsection{RPLU for Cauchy-like matrices}
A matrix $\mat{A} = (a_{i,j}) \in \mathbb{C}^{n \times m}$ is called Cauchy-like if its entries can be written as
\begin{equation}~\label{eq:cauchy-like}
a_{i,j} = \sum_{l = 1}^p  \frac{g_{i,l} b_{l,j}}{x_i - y_j},
\end{equation}
where $x_i, y_j, g_{i,l}, b_{l,j} \in \mathbb{C}$ and $x_i \neq y_j$ for all $i,j$.
Cauchy-like matrices arise in many applications, such as the discretization of integral operators and the solution of partial differential equations, in the solving of matrix equations~\cite{druskin2011analysis}, and in the computation of rational approximations of functions~\cite{antoulas2010interpolatory, aaa}.
Cauchy-like matrices always satisfy a Sylvester matrix equation of the following form: 
\begin{equation} \label{eq:cauchy-like-Sylvester}
\mat{X}\mat{A} - \mat{A}\mat{Y} = \mat{G}\mat{B}
\end{equation}
where $\mat{X} = \diag(x_1, \cdots, x_n) \in \mathbb{C}^{n \times n}$ and $\mat{Y} = \diag(y_1, \cdots, y_m) \in \mathbb{C}^{m \times m}$, and $\mat{G} = (g_{i,l}) \in \mathbb{C}^{n \times p}$ and $\mat{B} = (b_{l,j}) \in \mathbb{C}^{p \times m}$, with $p$ being a very small number (e.g., $< 5$) called the displacement rank of $\mat{A}$.  The matrices $\mat{G}$ and $\mat{B}$ are called the generators of the Cauchy-like matrix $\mat{A}$.  
Other structured matrix families, such as Vandermonde, Toeplitz, and Hankel matrices, satisfy analogous Sylvester equations with low-rank right-hand sides; collectively, these are known as matrices of low displacement rank. An excellent survey on displacement structure is~\cite{kailath1995displacement}.

A key property of Cauchy-like matrices is that their Schur complements are also Cauchy-like~\cite{gohberg1995fast,heinig2012inversion}; this closure does not hold for general low-displacement-rank matrices.
Concretely, let $(i,j)$ be a pivot with $a_{ij}\neq 0$ and consider the one-step residual update in~\cref{eq:residual_update}:
\[
\mat{A}^{(1)} := \mat{A} - \frac{1}{a_{ij}}\,\mat{A}_{:,j}\,\mat{A}_{i,:}.
\]
If $\mat{A}$ satisfies~\cref{eq:cauchy-like-Sylvester}, then $\mat{X}\mat{A}^{(1)}-\mat{A}^{(1)}\mat{Y} = \mat{G}^{(1)}\mat{B}^{(1)}$, with $\mat{G}^{(1)}$ and $\mat{B}^{(1)}$ given by:
\begin{align}
\mat{G}^{(1)} &:= \mat{G} - \frac{1}{a_{ij}}\,\mat{A}_{:,j}\,\mat{G}_{i,:},
\label{eq:generator_update_g} \\
\mat{B}^{(1)} &:= \mat{B} - \frac{1}{a_{ij}}\,\mat{B}_{:,j}\,\mat{A}_{i,:},
\label{eq:generator_update_b}
\end{align}
This is easily verified by direct computation. Note that $\mat{G}\mat{B}_{:,j}=(X-y_j I)\mat{A}_{:,j}$, $\mat{G}_{i,:}\mat{B}=\mat{A}_{i,:}(x_i I-Y)$, and $\mat{G}_{i,:}\mat{B}_{:,j}=(x_i-y_j)a_{ij}$, and thus:
\begin{align*}
\mat{G}^{(1)}\mat{B}^{(1)}
&= \mat{G}\mat{B} - \frac{1}{a_{ij}}\bigl(\mat{G}\mat{B}_{:,j}\mat{A}_{i,:} + \mat{A}_{:,j}\mat{G}_{i,:}\mat{B}\bigr) + \frac{1}{a_{ij}^2}\,\mat{A}_{:,j}\bigl(\mat{G}_{i,:}\mat{B}_{:,j}\bigr)\mat{A}_{i,:} \\
&= \mat{G}\mat{B} - \frac{1}{a_{ij}}\bigl(\mat{X}\mat{A}_{:,j}\mat{A}_{i,:} - \mat{A}_{:,j}\mat{A}_{i,:}\mat{Y}\bigr)
 = \mat{X}\mat{A}^{(1)}-\mat{A}^{(1)}\mat{Y}.
\end{align*}
In particular, the displacement rank (dimension of the generators) does not increase, and the generators of $\mat{A}^{(1)}$ can be computed from the generators of $\mat{A}$ using~\cref{eq:generator_update_g} and~\cref{eq:generator_update_b} in $\mathcal{O}((n+m)r)$ operations.
Since each residual $\mat{A}^{(k)}$ produced by RPLU is Cauchy-like, one can carry out the iteration and compute the residual implicitly by simply updating the generators $\mat{G}^{(k)}$ and $\mat{B}^{(k)}$. 

The remaining operation required to carry out RPLU is to sample a pivot pair $(i,j)$.
As highlighted in~\Cref{sec:practical_algorithm}, this can be done efficiently if the row norms of $\mat{A}^{(k)}$ can be computed efficiently.
Using~\cref{eq:cauchy-like}, these row norms can be written as
\[
\|\mat{A}^{(k)}_{i,:}\|_2^2 = \sum_{j = 1}^m \frac{1}{|x_i - y_j|^2} \left|\sum_{l = 1}^p g_{i,l} b_{l,j}\right|^2.
\]
This vector can be computed to accuracy $\epsilon$ in $\mathcal{O}(p^2 (m+n)\log(1/\epsilon))$ operations using a fast multipole method~\cite{greengard1987accelerating}, or to lower accuracy using a Barnes-Hut approximation~\cite{barnes1986hierarchical}.
However, high-accuracy computation is not necessary to sample pivots from the desired distribution: even a crude upper bound can be used to sample the distribution exactly and efficiently using rejection sampling.
Specifically, we compute upper bounds $u_i^{(k)}$ satisfying
\begin{align*}
  \|\mat{A}^{(k)}_{i,:}\|_2^2 \leq u_i^{(k)} \leq \nu \|\mat{A}^{(k)}_{i,:}\|_2^2
\end{align*}
for a fixed $\nu > 1$ (in experiments, we pick $\nu = 5$), using a tree-based, Barnes-Hut-like computation detailed in~\Cref{sec:cauchy-row-norm-bound}.

The resulting algorithm is summarized in~\Cref{alg:lu-cur-cauchy}. 
The algorithm returns only the selected index sets $\I_K$ and $\J_K$. If a CUR factorization is needed, the core matrix $\mat{U}\in\mathbb{C}^{K\times K}$ can be formed afterward in $\mathcal{O}(K^3)$ operations. Alternatively, one may store the rank-$1$ update data $(\mat{c}/a,\mat{r})$ from each iteration to recover an explicit LU factorization.
We also implement an approximate version of C2PLU using the same upper bound: at each step, select the row maximizing the bound, then select the column with the largest entry in that row. This heuristic has no theoretical guarantees, but performs well in practice.

\begin{algorithm}[H]
  \caption{RPLU for Cauchy-like matrices}
  \label{alg:lu-cur-cauchy}
  \begin{algorithmic}
    \State \textbf{Input:} Generators $\mat{G}^{(0)} \in \mathbb{C}^{n \times p}$, $\mat{B}^{(0)} \in \mathbb{C}^{p \times m}$, vectors $\mathbf{x} \in \mathbb{C}^n$, $\mathbf{y} \in \mathbb{C}^m$, target rank $K$, factor $\nu \geq 1$
    \State Precompute interaction lists as in~\Cref{alg:gbub-precompute} \Comment{$\mathcal{O}(n\log(n) + m\log(m))$ setup}
    \State Initialize $\I_0 \leftarrow \emptyset$, $\J_0 \leftarrow \emptyset$
    \For{$k = 0$ to $K-1$}
      \State \textbf{Step 1: Rejection-sample pivot row}
      \State Compute upper bounds $u_i^{(k)}$ for all $i \in [n]$ using~\Cref{alg:gbub-online}
      \Repeat
        \State Sample $i \sim u_i^{(k)} / \sum_j u_j^{(k)}$
        \State $\mat{r} \gets \bigl(\mat{G}^{(k)}_{i,:} \mat{B}^{(k)}\bigr) \oslash (x_i - \mathbf{y})$ \Comment{Row of $\mat{A}^{(k)}$}
        \State $\rho \gets \|\mat{r}\|_2^2$
        \State Accept with probability $\rho / u_i^{(k)}$
      \Until{accepted}
      \State $i_{k+1} \gets i$

      \State \textbf{Step 2: Sample pivot column}
      \State Sample $j_{k+1} \sim |\mat{r}_j|^2 / \|\mat{r}\|_2^2$
      \State $a \gets \mat{r}_{j_{k+1}}$
      \State $\mat{c} \gets \bigl(\mat{G}^{(k)} \mat{B}^{(k)}_{:,j_{k+1}}\bigr) \oslash (\mathbf{x} - y_{j_{k+1}})$ \Comment{Column of $\mat{A}^{(k)}$}
      \State \textbf{Step 3: Update generators}

      \State $\mat{G}^{(k+1)} \gets \mat{G}^{(k)} - \frac{1}{a} \mat{c} (\mat{G}^{(k)}_{i_{k+1},:})^\top$
      \State $\mat{B}^{(k+1)} \gets \mat{B}^{(k)} - \frac{1}{a} \mat{B}^{(k)}_{:,j_{k+1}} \mat{r}^\top$
      \State $\I_{k+1} \gets \I_k \cup \{i_{k+1}\}$, $\J_{k+1} \gets \J_k \cup \{j_{k+1}\}$
    \EndFor
    \State \textbf{Return} $\I_K$, $\J_K$

  \end{algorithmic}

\end{algorithm}

\subsection{Numerical results} \label{sec:cauchy_numerical}
We benchmark RPLU and C2PLU on Loewner matrices, which are Cauchy-like matrices of the form:
\begin{equation}\label{eq:aaa-loewner}
  \mat{L}_{i,j}=\frac{f^{x}_i-f^{y}_j}{x_i-y_j}.
\end{equation}
These matrices have generators:
\begin{equation} \label{eq:loewner-generators}
  \mat{G}=\begin{bmatrix} f^{x}_i/\alpha & \alpha \end{bmatrix}_i\in\mathbb{C}^{n\times 2},
  \qquad
  \mat{B}=\begin{bmatrix} \alpha \\ -f^{y}_j/\alpha \end{bmatrix}_j\in\mathbb{C}^{2\times m},
\end{equation}
where $\alpha$ is a scaling factor chosen for numerical stability as
\[
\alpha \;=\; \sqrt{\max\!\left\{\max_{1\le i\le n}|f^{x}_i|,\ \max_{1\le j\le m}|f^{y}_j|\right\}}.
\]
These matrices arise in the computation of rational approximations as discussed in the next section.
We compare against randomized SVD with FMM-accelerated matrix-vector products~\cite{greengard1987accelerating}\footnote{Our implementation uses FMMlib2D implementation, see~\url{https://github.com/zgimbutas/fmmlib2d} and Python bindings \url{https://github.com/dbstein/pyfmmlib2d}.}, achieving complexity $\mathcal{O}((n+m)k^2 + (n+m)k \log(1/\epsilon))$ where $k$ is the target rank and $\epsilon$ is the FMM accuracy (set to $5\times10^{-15}$). All experiments are performed on CPU on an Apple M3 Max (16 cores) with 64\,GB RAM.

\begin{figure}[htbp]
  \centering
  \adjustbox{width=1.0\textwidth}{\includegraphics{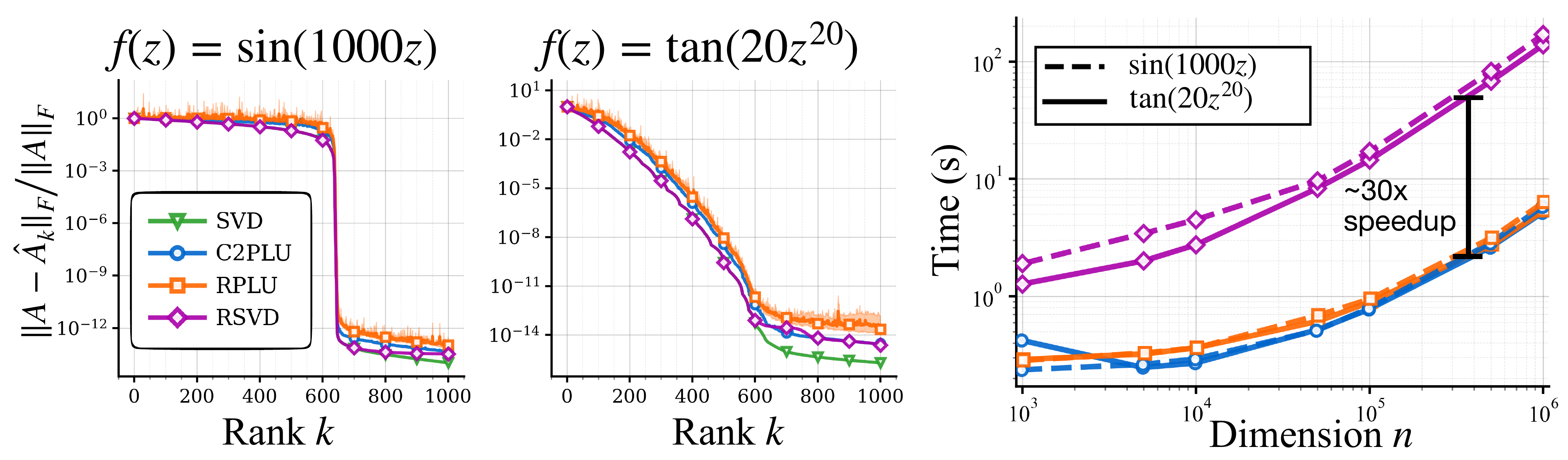}}
  \caption{Approximation quality and timing of RPLU, C2PLU, and randomized SVD on two classes of Loewner matrices.
  Left: approximation quality with $n=m=2000$ points for $f(z)=\sin(1000z)$ with $x_i$ and $y_i$ sampled uniformly on $[-1,1]$, and for $f(z)=\tan(20z^{20})$ with points sampled uniformly on the unit disk $\{z\in\mathbb{C}: |z|\le 1\}$. RPLU is repeated 10 times; we report the mean, with a shaded band showing the min/max over the 10 runs and the randomized SVD is computed once at rank $1000$. Right: scaling of computation time with $n=m$ to compute a rank-$k=1000$ approximation.}
  \label{fig:Loewner_scaling}
\end{figure}

Results in~\Cref{fig:Loewner_scaling} show that RPLU and C2PLU are approximately $30\times$ faster than randomized SVD for large $n$ (here, $n \geq 1000$).\footnote{We emphasize that unlike the results in \Cref{sec:practical_algorithm}, the algorithms here are competitive for even moderate $n$ (e.g., $n \approx 1000$) and regardless of memory footprint.}
This speedup reflects a fundamental difference: randomized SVD requires high-accuracy FMM evaluations for each matrix-vector product, while RPLU needs only crude upper bounds on row norms for rejection sampling, which can be computed much more efficiently using the tree-based algorithm in~\Cref{sec:cauchy-row-norm-bound}.
All three methods produce approximations within a factor of $10$ of the optimum at intermediate ranks ($k \leq 600$); the randomized SVD matches the SVD nearly exactly. At higher ranks, the randomized SVD appears slightly more stable than RPLU: for example, at $k=1000$ the relative residual is $2\times 10^{-15}$ for randomized SVD versus $2\times 10^{-14}$ for RPLU.

An additional advantage of RPLU and C2PLU is that they return CUR decompositions, which we exploit in the next section.

\subsection{Fast rational approximation} \label{sec:fast_rational}

A natural application where Cauchy-like matrices appear is in the problem of rational approximation from function samples.
Let $f:\Omega\subset\mathbb{C}\to\mathbb{C}$ and suppose we are given values $f_i = f(z_i)$ at sample points $\mathcal{Z} = \{z_1,\ldots,z_m\}\subset\Omega$.
The goal is to construct a rational function $r(z)=p(z)/q(z)$, where $p(z)$ and $q(z)$ are polynomials, that approximates $f$ on $\mathcal{Z}$ (or more generally on $\Omega$).

A popular method for rational approximation is the AAA (adaptive Antoulas--Anderson) algorithm~\cite{aaa}, which constructs a sequence of rational functions in the following barycentric form:
\begin{equation}\label{eq:aaa-barycentric}
  r_k(z)
  =
 \left(\sum_{j=1}^k \frac{w_j^{(k)} f(t_j)}{z-t_j}\right)
  \Bigg/
  \left(\sum_{j=1}^k \frac{w_j^{(k)}}{z-t_j}\right).
\end{equation}
Here, the support points $\mathcal{T}^{(k)} = \{t_j\}_{j=1}^k$ are chosen from $\mathcal{Z}$, and one can show that for each $1 \leq j \leq k$, $\lim_{t \rightarrow t_j} r_k(t) = f(t_j)$. AAA constructs $r_k$ iteratively by selecting one support point at a time and recomputing the weights $w^{(k)}=(w_j^{(k)})_{j=1}^k$ at each iteration. 
At iteration $k$, if $\max_{z_i \in \mathcal{Z}\setminus \mathcal{T}^{(k-1)}} |f(z_i)-r_{k-1}(z_i)|$ exceeds a prescribed tolerance $\epsilon$, AAA selects the next support point
\[
t_{k} = \argmax_{z_i \in \mathcal{Z}\setminus \mathcal{T}^{(k-1)}} |f(z_i) - r_{k-1}(z_i)|,
\]
that is, the point at which $r_{k-1}$ attains its largest error. This can be interpreted as a greedy pivoting strategy for selecting interpolating points. Once the $k$th support point is selected, the weights must be updated. 
Given a fixed set of support points $\mathcal{T}^{(k)}$,  weights $w^{(k)}\in\mathbb{C}^k$ are chosen as the entries of the smallest right singular vector of the Loewner matrix $\mat{L}^{(k)} \in \mathbb{C}^{(m-k) \times k}$, with entries
\begin{equation} \label{eq:Laaa}
\mat{L}^{(k)}_{i,j} = \frac{f(z_i)-f(t_j)}{z_i-t_j},
\qquad z_i \in \mathcal{Z}\setminus \mathcal{T}^{(k)}.
\end{equation}
This choice is motivated by the linearization:
\begin{align}\label{eq:linearization-aaa}
\sum_{j=1}^k  \frac{ w_j f(t_j)}{z_i - t_j} \big/ \sum_{j=1}^k \frac{ w_j}{z_i - t_j}  \approx f(z_i)
\quad\Longleftrightarrow\quad 
\sum_{j=1}^k w_j \frac{ f(t_j)}{z_i - t_j}   \approx \sum_{j=1}^k w_j \frac{ f(z_i)}{z_i - t_j} 
\quad\Longleftrightarrow\quad 
\mat{L}^{(k)} w \approx 0.
\end{align}
We supply pseudocode in~\Cref{alg:fast-rational-approximation} for convenience, and refer the reader to the recent survey in~\cite{nakatsukasa2025applications} for a thorough review of many variations and additional developments associated with the AAA algorithm.

The most costly stage of the AAA algorithm is the repeated solving of the weight optimization problem, which requires computing the SVD of the Loewner matrix at each iteration.  We propose a CUR-based alternative (CUR-AAA), where the pivoting stage and selection of $k$ support points is completed via a CUR decomposition of a Loewner matrix related to the data. Once a collection of candidate support points are chosen, an SVD is applied to compute the weights for the barycentric interpolant at the given support points. This cuts down substantially on the overall cost of the algorithm and we find that it is especially useful in settings where high-degree rational functions and many samples in the domain are required.

The CUR-AAA method begins by partitioning the $m$ sample points $\mathcal{Z}$ randomly into two sets: $\mathcal{X}$ and $\mathcal{Y}$. These two sets define a large (nearly square) Loewner matrix $\mat{\tilde{L}}\in\mathbb{C}^{\lceil m/2\rceil\times\lfloor m/2\rfloor}$ with entries
\[
\mat{\tilde{L}}_{ij} \;=\; \frac{f(x_i)-f(y_j)}{x_i-y_j},
\qquad x_i\in\mathcal{X},\ \ y_j\in\mathcal{Y}.
\]
We then use \Cref{alg:lu-cur-cauchy} to compute a CUR approximation of $\mat{\tilde{L}}$ using RPLU (or C2PLU), so that $\mat{\tilde{L}}\approx \mat{C}\mat{U}\mat{R}$. The selected row or column indices can be chosen as support points: we take whichever set minimizes the interpolation error over the remaining points. CUR-AAA is summarized in~\Cref{alg:rplu-rational-approximation}.

\begin{algorithm}[H]
  \caption{CUR-AAA using RPLU}
  \label{alg:rplu-rational-approximation}
  \begin{algorithmic}
    \State \textbf{Input:} sample points $\mathcal{Z}=\{z_1,\ldots,z_m\}$, values $f=(f_1,\ldots,f_m)$, tolerance $\epsilon$
    \State Partition $\mathcal{Z}$ randomly into $\mathcal{X}$ and $\mathcal{Y}$, and split $f$ accordingly into $f^{x}$ and $f^{y}$.
    \State Form the generators $\mat{G},\mat{B}$ for the Loewner matrix $\mat{\tilde{L}}$ as in~\eqref{eq:loewner-generators}.
    \State Obtain a CUR using~\Cref{alg:lu-cur-cauchy} until $  \|u^{(k)}\|_{\infty} \le \epsilon^2 \|u^{(0)}\|_{\infty} $.
    \State Fit rational approximations $r_{\I}$ and $r_{\J}$ using $\I$ and $\J$ as support points, respectively, by forming $\mat{L}$ as in \cref{eq:Laaa} and solving for the barycentric weights. 
    \If{$\max_{z_i \in \mathcal{Z}} |f_i - r_{\I}(z_i)| < \max_{z_i \in \mathcal{Z}} |f_i - r_{\J}(z_i)|$}
      \State \textbf{return} $r_{\I}$
    \Else
      \State \textbf{return} $r_{\J}$
    \EndIf
  \end{algorithmic}
\end{algorithm}

A rigorous analysis of this algorithm is beyond the scope of this work; for now, we provide intuition by analogy with AAA. AAA selects support points from $\mathcal{Z}$, thereby choosing a basis $\{1/(z-t_j)\}$ in which the remaining samples are well-approximated. AAA terminates when the approximation error is small, which (via the linearization in~\cref{eq:linearization-aaa}) is closely tied to the Loewner matrix developing a small singular value.
The CUR-based approach also seeks a basis, but it uses a subset of $\mathcal{X}$ to fit the samples in $\mathcal{Y}$ and a subset of $\mathcal{Y}$ to fit the samples in $\mathcal{X}$. CUR naturally computes a discrete set of basis functions from which remaining non-selected rows (or columns) can be well-represented. 
When there are sufficiently many samples, this factor-of-$2$ reduction in the number of training samples appears to have negligible effect and still leads to an effective basis. Since the large, square Loewner matrix $\mat{L}$ has rapidly decaying singular values, an accurate CUR decomposition is obtained when $\mat{C}$ or $\mat{R}$ (analogous to the rectangular Loewner matrix $\mat{L}^{(k)}$ in AAA) also has a small singular value.

The computational cost of AAA is dominated by repeated SVD computations of the Loewner matrix, which cost $\mathcal{O}(nk^2)$ per iteration, for a total cost of $\mathcal{O}(nk^3)$ over $k$ iterations.
The cost of CUR-AAA is $\mathcal{O}(n\log(n)k)$ for the CUR step and $\mathcal{O}(nk^2)$ for the SVD step, for a total of $\mathcal{O}(nk^2 + n\log(n)k)$.\footnote{We note that for large $k$, the SVD computations in both methods may be accelerated using Lanczos iterations with FMM-accelerated matrix-vector products. We do not explore this here.}

\Cref{fig:aaa_scaling} shows the scaling behavior of AAA and CUR-AAA computed using RPLU and C2PLU for the family $f(z)=\tan(\sqrt{d}\,z^{\sqrt{d}})$ as $d^2$ ranges from $1$ to $20$. This family admits accurate rational approximants of degree $k = \mathcal{O}(d)$ and has $\mathcal{O}(d)$ poles inside the unit disk. Sampling only on the boundary of the disk is not enough to recover all the poles, so the sampling grid must include points inside the disk.
The complexity analysis in this regime is pessimistic: empirically, the runtime scales as $\mathcal{O}(d)$ for CUR-AAA and $\mathcal{O}(d^2)$ for AAA\footnote{We use the baryrat implementation of AAA~\cite{hofreither2021algorithm} in Python.}, as the SVD scaling is $\mathcal{O}(nd)$ for small $d$.
Both algorithms are run with relative tolerance $\epsilon=10^{-11}$.\footnote{These tolerance criteria are not the same. AAA terminates when the relative interpolant error is less than $\epsilon$, while CUR-AAA terminates when the relative upper bound on the row norms drops below $\epsilon$, since it does not compute the interpolant until the end.}
The approximation error over $\mathcal{Z}$ is always smaller for AAA by a factor of $10$ in most cases\footnote{We note that these errors could be decreased by reducing the tolerance parameter $\epsilon$, although this creates numerical issues for all three algorithms at large values of $d$.}. The error on the hold-out grid is comparable across all three methods for $d \geq 100$, although RPLU appears less stable.

\begin{figure}[htbp]
  \centering
  \adjustbox{width=1.0\textwidth}{\includegraphics{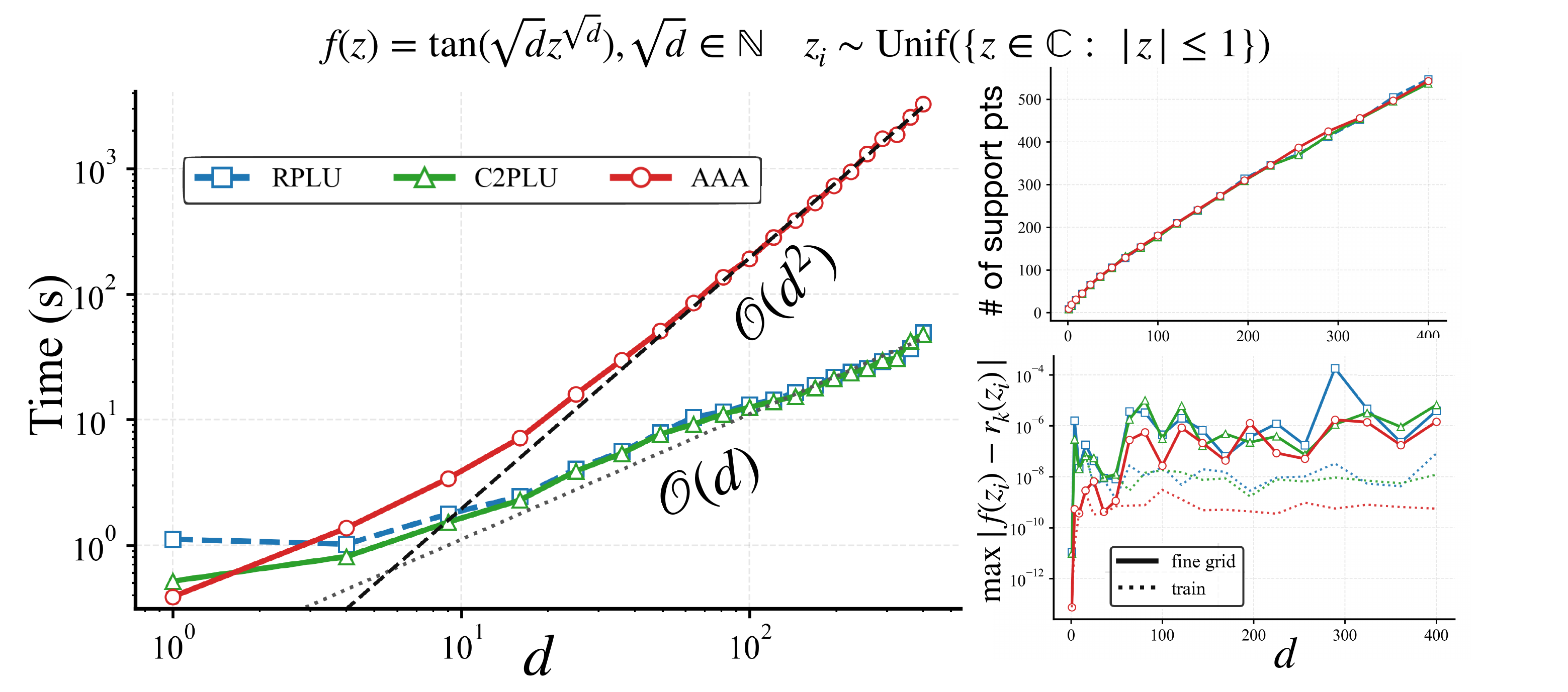}}
  \caption{Scaling analysis of AAA and CUR-AAA (using RPLU and C2PLU). In each case, $200{,}000$ sample points are drawn uniformly from the unit disk, and the algorithms are run until a relative tolerance of $\epsilon=10^{-11}$ is achieved. In most cases, all three algorithms return a similar number of support points (top right), growing roughly linearly with $d$. The maximum interpolation error (bottom left) is shown on the training set $\mathcal{Z}$ and on a hold-out Cartesian grid with $2{,}000{,}000$ points inside the unit disk.}
  \label{fig:aaa_scaling}
\end{figure}

\Cref{fig:aaa_comp_400} shows the interpolant computed by each method for $d=400$, which takes about 50 seconds to compute for CUR-AAA, and about 50 minutes for AAA.
All three methods locate each of the 240 poles inside the unit disk: RPLU and C2PLU locate each pole to accuracy better than $2\times 10^{-12}$, and AAA locates each pole to accuracy $8\times 10^{-12}$.
However, the RPLU interpolant also exhibits two spurious poles inside the unit disk that are not present in the other two interpolants. These spurious poles can be removed by reducing the tolerance to $\epsilon=10^{-10}$.

Overall, we find that CUR-AAA computes high-accuracy interpolants at a fraction of the cost of AAA thanks to the fast CUR decomposition step, although AAA appears to be more robust in our current implementation.

\begin{figure}[t!]
  \centering
  \includegraphics[width=0.7\textwidth,trim={0cm 0cm 0 1cm},clip]{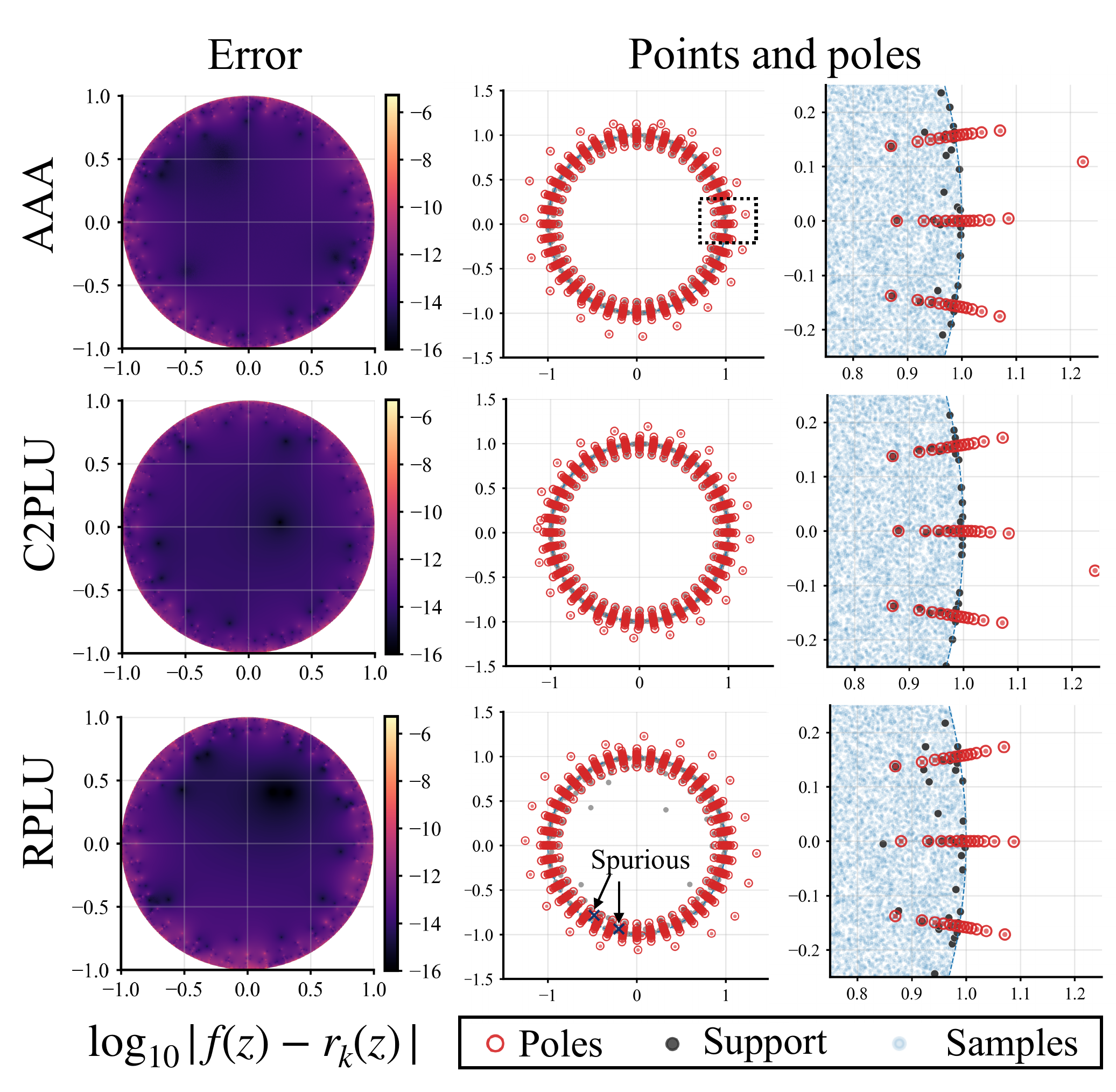}
  \caption{Poles and support points for rational approximants of $f(z)=\tan(20 z^{20})$ constructed from $m=200{,}000$ samples drawn uniformly at random from the unit disk. Rows correspond to AAA (top), CUR-AAA with C2PLU (middle), and CUR-AAA with RPLU (bottom). Red circles mark poles of the approximants, black points mark selected support points, and blue points show samples; the right column is a zoom of the dashed region.}
  \label{fig:aaa_comp_400}
\end{figure}

\subsection{Comparison with factored ADI and extension to other matrices with displacement structure}
A particularly well-suited low-rank approximation method for matrices with low displacement rank (including Cauchy-like matrices) is the factored alternating-direction implicit (ADI) method~\cite{benner2009adi,li2002low}. We briefly discuss factored ADI and compare it with RPLU.

Matrices with low displacement rank can be characterized by a Sylvester displacement equation of the form
\begin{align*}
\mat{M}\mat{A} - \mat{A}\mat{N} = \mat{G}\mat{B},
\end{align*}
where $p= \operatorname{rank}(\mat{G}\mat{B})$ is the displacement rank of $\mat{A}$. Factored ADI is an iterative algorithm that constructs low rank approximations to $\mat{A}$ using the so-called coefficient matrices $\mat{M}$ and $\mat{N}$ as well as the generators $\mat{G}$ and $\mat{B}$. Given shift parameters $\{(\alpha_j,\beta_j)\}_{j=1}^k$, factored ADI iteratively applies shifted inverts and matrix products to  blocks of $p$ columns at each iteration. A full description of the algorithm can be found in~\cite{benner2009adi}. The shifted products and solves are of the form $(\mat{M}-\alpha_j I )(\mat{M}-\alpha_{j+1}\mat{I})^{-1}$ and $(\mat{N}^*-\overline{\beta}_jI)(\mat{N}^*-\overline{\beta}_{j+1}I)^{-1}$. After $k$ iterations, factored ADI produces a rank-$(\le pk)$ approximation,
\[
\mat{A} \approx \mat{Z}_{k}\mat{W}_{k}^*, \qquad \mat{Z}_{k} \in \mathbb{C}^{n \times pk}, \ \mat{W}_{k} \in \mathbb{C}^{m \times pk}.
\]
In the Cauchy-like case, $\mat{M}=\diag(\mathcal{X})$ and $\mat{N}=\diag(\mathcal{Y})$, so each step of factored ADI  reduces to applying four diagonal scalings, each to a set of $p$ columns. When the sets $\mathcal{X} = \{x_1, \cdots, x_n\}$ and $\mathcal{Y} = \{y_1, \cdots, y_m\}$ are well-separated in the complex plane, an optimal choice of the shift parameters yields an error bound of the form~\cite{beckermann2019bounds}
$$ \| \mat{A} - \mat{Z}_k\mat{W}_k^*\|_2 \leq C \rho^{\lfloor k/r \rfloor},$$
where $\rho < 1$ is dictated by the condenser capacity of the sets $\mathcal{X}$ and $\mathcal{Y}$, and $C$ depends on the geometry of sets enclosing $\mathcal{X}$ and $\mathcal{Y}$. Thus, larger separation between $\mathcal{X}$ and $\mathcal{Y}$ leads to smaller $\rho$ and faster convergence.

For Cauchy-like matrices with displacement rank $p = 1$, ADI is observed to often produce near-best low-rank approximations. For $p > 1$, however, ADI can be noticeably suboptimal; see~\Cref{fig:fadi_comparison}. In these cases RPLU typically achieves better approximation quality for a given rank.
Selecting the $2k$ shift parameters needed to attain the optimal convergence rate is another practical hurdle for factored ADI. These shifts depend on $\mathcal{X}$ and $\mathcal{Y}$ and therefore require some a priori description of sets enclosing them. When $\mathcal{X}$ and $\mathcal{Y}$ lie in disjoint disks or intervals, the shifts can be computed with negligible cost~\cite{townsend2018singular}, and the factors can then be formed in $\mathcal{O}((m+n)kp)$ operations. 

For more general configurations, computing the shifts requires a precomputation step involving the poles and zeros of a minimax rational function, with a cost governed by the geometry of the sets containing $\mathcal{X}$ and $\mathcal{Y}$. This step can be difficult when the sets are complicated or unknown. If the sets are overlapping (as in~\cref{sec:cauchy_numerical} and~\cref{sec:fast_rational}), then one must select shift parameters using only the discrete points in $\mathcal{X}$ and $\mathcal{Y}$, and the convergence rate of ADI is expected to be slow. Furthermore, factored ADI does not directly return a CUR factorization. Its output can be post-processed to produce an interpolative CUR approximation in $\mathcal{O}((m+n)r^2+ r^3)$ operations, where $r=kp$ is the rank of the approximation, by applying CPQR to each of the factors in the approximant to subselect rows and columns of $\mat{A}$. 

Factored ADI can be applied to many low-displacement-rank families beyond Cauchy-like matrices, including Vandermonde, Toeplitz, Hankel, Chebyshev--Vandermonde matrices, and more~\cite{beckermann2019bounds,kailath1995displacement}. Extending RPLU beyond the Cauchy-like setting is less direct: for other families of matrices with low displacement rank, the Schur complement retains the algebraic structure of the original matrix only if no pivoting is employed. This makes it impossible to run RPLU in the same manner as in~\Cref{alg:lu-cur-cauchy}, at least naïvely.

However, as discussed in~\cite{heinig2012inversion, pan2000superfast}, one can often convert these matrices into factorizations involving Cauchy-like matrices. For example, consider a Vandermonde matrix $\mat{V}$ with entries $\mat{V}_{jk} = x_j^{k-1}$ for points $X:=\{ x_1, x_2, \cdots, x_m\} \subset \mathbb{C}$. This matrix is numerically low rank whenever $X$ lies on the real line~\cite{beckermann2019bounds} or sufficiently far from the unit circle~\cite{rubin2022bounding}. It satisfies the Sylvester matrix equation
\begin{equation}  \label{eq:vand}
\mat{X} \mat{V} - \mat{V} \mat{Q} = \mat{g} \mat{e}_n^*,  
\end{equation}
where $\mat{g}$ is an $m \times 1$ column vector with $g_{j,1} = x_j^{n-1}\!-\!1$, $\mat{X} = {\rm diag}(x_1, \cdots, x_m)$, and $\mat{Q} \in \mathbb{C}^{n \times n}$ is the permutation matrix that circularly shifts columns to the left. The eigendecomposition of $\mat{Q}$ is known explicitly: letting $\omega = \exp(2 \pi i/n)$, $\mat{Q} = \mat{F}^* \mat{\Lambda} \mat{F}$, where $\mat{F}_{jk} = \omega^{j(2k-1)}/\sqrt{n}$ is a discrete Fourier transform matrix and $\mat{\Lambda} = {\rm diag}(\omega^2, \omega^4, \cdots, \omega^{2n}).$ Applying $\mat{F}^*$ on the left of \cref{eq:vand}, we find that $\mat{V}$ can be expressed as 

\begin{equation} \label{eq:vandcauchy}
\mat{ V} =  {\rm diag}(\mat{g}) \mat{C}  {\rm diag}(\mat{e}_n^T\mat{F}^*) \mat{F} ,
 \end{equation}
where $\mat{C}$ is Cauchy and has entries $\mat{C}_{jk} = 1/(x_j - \omega_k)$, and $\mat{A} = {\rm diag}(\mat{g}) \mat{C}  {\rm diag}(\mat{e}_n^T\mat{F}^*) $ is Cauchy-like. Constructing a low rank approximation to $\mat{A}$ via \Cref{alg:lu-cur-cauchy} immediately yields a low rank approximation to $\mat{V}$ (though it is not a CUR).
Analogous reductions exist for Toeplitz and Hankel matrices~\cite{gohberg1995fast} that can be useful in fast algorithms for solving linear systems; it is known that any Toeplitz or Hankel matrix, whether it is low rank or not, can be transformed into a Cauchy-like matrix with hierarchical low rank structure~\cite{beckermann2025compression, heinig2012inversion}. Applying RPLU in the hierarchical setting is possible but outside the scope of this paper.

\begin{figure}[t!]
  \centering
  \adjustbox{width=1.0\textwidth}{\includegraphics[trim={0cm 4.5cm 0 5cm},clip]{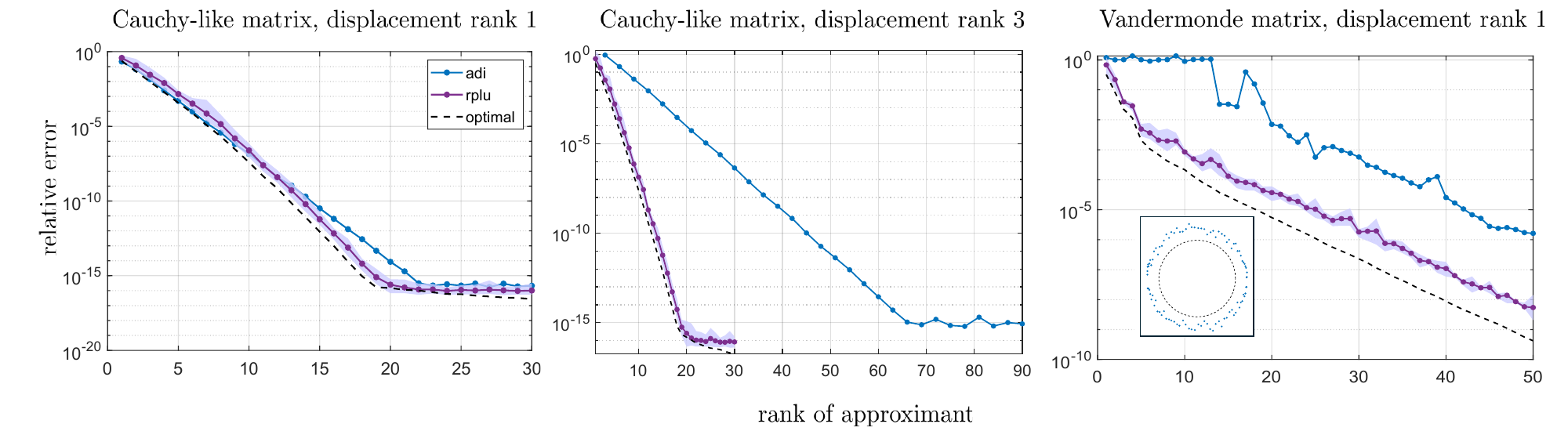}}
  \caption{Approximation quality of RPLU (purple) compared with factored ADI (blue) and the SVD (black dotted) vs rank of constructed low rank approximation. For the Cauchy-like matrices (left and center), $\mathcal{X}$ consists of $100$ equally-spaced samples from  $[-1, -.01]$ and $\mathcal{Y} = -\mathcal{X}$.  For the Vandermonde matrix (right), $\mathcal{X}$ consists of $100$ points sampled in a circle of radius $1.3$ and then jittered in a random direction at a magnitude of at most $.08$ (points shown at bottom right). We apply RPLU to the Cauchy-like matrix $\mat{A}$ in \cref{eq:vandcauchy}, which is then used to construct the low rank approximation (purple). Factored ADI is applied by first computing shift parameters using the discrete sets of points $\mathcal{X}$ and the shifted roots of unity that form the eigenvalues of $\mat{F}$.  In all plots, RPLU is repeated $20$ times and the mean performance is shown as a solid line. The shaded region shows variance across runs.}
  \label{fig:fadi_comparison}
\end{figure}

\section{Conclusion} \label{sec:conclusion}

This work introduced Randomly Pivoted LU (RPLU), a randomized variant of Gaussian elimination with complete pivoting for computing low-rank approximations. 
We proved that when the singular values of $\mat{A}$ decay geometrically at rate $\rho<1/2$, the RPLU iterates converge geometrically at rate $(2\rho)^k$; in practice, we often observe substantially faster convergence.
We also advocated complete $2$-norm pivoting (C2PLU) as a practical heuristic: while it can perform much worse than RPLU in the presence of outliers (see~\cref{fig:rplu_intro}), it is highly effective in many applications and often outperforms RPLU.

We demonstrated two practical settings in which RPLU and C2PLU are particularly effective. For very large structured matrices, their small memory footprint enables high-rank approximations on GPUs. For Cauchy-like matrices, displacement structure allows fast Schur complement updates and efficient computation.
In both cases, these factorizations serve as building blocks for faster downstream algorithms, including preconditioners for large linear systems and fast rational approximation.

Several questions remain open. First, an oversampling bound for RPLU, analogous to~\cref{theorem:rpchol_bound} for RPCholesky, may clarify when the randomized strategy is provably near-optimal without assuming geometric decay. More broadly, sharper theory is needed to narrow the gap between provable guarantees and the performance observed in practice for both RPLU and RPCholesky. Empirically, C2PLU often outperforms RPLU, mirroring observations made by some other authors in the PSD case~\cite{fornace2024column}; since worst-case analyses for deterministic pivoting are often pessimistic, a natural next step is to identify regimes in which such greedy schemes can be proved near-optimal.

\section*{Acknowledgements}
We thank Ethan Epperly, Joel Tropp and Elizaveta Rebrova for helpful discussions and suggestions.
\newpage
\appendix

\numberwithin{equation}{section}
\numberwithin{figure}{section}
\numberwithin{table}{section}
\numberwithin{algorithm}{section}

\section{Proofs}

\subsection{Proof of~\cref{theorem:doubling}} \label{sec:doubling_proof}
\begin{proof}
  This proof follows closely the argument used to prove the doubling lemma (\cref{lem:rchol_doubling}) for RPCholesky in~\cite{rpchol}.

  Let $\sigma_j$ denote the $j$-th singular value of $\mat{A}$. Define $\mat{G}\coloneqq \mat{A}^*\mat{A}$ and $\mat{G}^{(1)}\coloneqq \mat{A}^{(1)*}\mat{A}^{(1)}$.
  By the Ky Fan variational principle,
  \begin{equation*}
  \tr\!\bigl(\mat{G}^{(1)}-\lowrank{\mat{G}^{(1)}}_{k-1}\bigr)
  =\min_{\substack{\vec u_k,\dots,\vec u_m\\ \text{orthonormal}}}\ \sum_{j=k}^{m}\vec u_j^{*}\mat{G}^{(1)}\vec u_j.
  \end{equation*}
  Choose \(\vec u_{k+1},\dots,\vec u_m\) to be the unit eigenvectors \(\vec v_{k+1},\dots,\vec v_m\) associated with the \(m-k\) smallest eigenvalues of \(\mat{G}\).
  The final vector \(\vec g=\vec u_k\) will be fixed later; for now, we observe that
  \begin{align}
  \EX\tr\!\bigl(\mat{G}^{(1)}-\lowrank{\mat{G}^{(1)}}_{k-1}\bigr)
  &\le \sum_{j=k+1}^{m}\vec v_j^{*}\EX[\mat{G}^{(1)}]\vec v_j \;+\; \EX\bigl[\vec g^{*}\mat{G}^{(1)}\vec g\bigr]\notag\\
  &\le \sum_{j=k+1}^{m}\vec v_j^{*}\,2\Phi(\mat{G})\,\vec v_j \;+\; \EX\bigl[\vec g^{*}\mat{G}^{(1)}\vec g\bigr],\label{eq:split-main}
  \end{align}
  where we have used the fact that \(\EX[\mat{G}^{(1)}]\preceq 2\Phi(\mat{G})\) (see~\cref{eq:G1_bound}) and the linearity of expectation.
  To bound the first term, we use the fact that \(\Phi(\mat{G}) \preceq \mat{G}\):
\begin{equation*}
2\sum_{j=k+1}^{m}\vec v_j^{*}\Phi(\mat{G})\,\vec v_j\;\le\;2\sum_{j=k+1}^{m} \vec v_j^{*}\mat{G}\vec v_j\;\le\;2\sum_{j=k+1}^{m} \sigma_j^{2} = 2\tr(\mat{G}-\lowrank{\mat{G}}_{k}).
\end{equation*}

  We now bound the second term. Denote \((r_1,c_1)\) as the first pivot; the first LU update is
  \[
  \mat{A}^{(1)} \;=\; \mat{A} \;-\; \frac{1}{a_{r_1c_1}}\;\mat{A}\mathbf e_{c_1}\,\mathbf e_{r_1}^{\top}\mat{A}.
  \]
  Its Gram matrix is
  \begin{align}
  \mat{G}^{(1)}
  &= \mat{A}^{(1)*}\mat{A}^{(1)}\notag\\
  &= \mat{A}^*\mat{A}
  - \frac{1}{a_{r_1c_1}}\;\mat{A}^*\mat{A}\mathbf e_{c_1}\,\mathbf e_{r_1}^{\top}\mat{A}
  - \frac{1}{\overline{a_{r_1c_1}}}\;\mat{A}^*\mathbf e_{r_1}\,\mathbf e_{c_1}^{*}\mat{A}^*\mat{A}
  + \frac{1}{|a_{r_1c_1}|^{2}}\,\|\mat{A}\mathbf e_{c_1}\|_2^2\,\mat{A}^*\mathbf e_{r_1}\,\mathbf e_{r_1}^{\top}\mat{A}.
  \label{eq:C1-expanded}
  \end{align}

  Let \(\mat{P}\) be the orthogonal projector onto the span of the top \(k\) eigenvectors of \(\mat{G}\), and set \(\mat{P}_\perp=\Id-\mat{P}\).
  We distinguish two cases depending on the first pivot column \(c_1\).

  \paragraph{Case 1: \(\mat{P}\mathbf e_{c_1}=0\).}
  Set \(\vec g=\vec v_1\), the top eigenvector of \(\mat{G}\), and let \(\vec w_1=\mat{A}\vec v_1/\sigma_1\) be the corresponding left singular vector.
  Using \eqref{eq:C1-expanded} and \(\vec v_1^{*}\mat{A}^*\mat{A}\mathbf e_{c_1}=\sigma_1^2\,\vec v_1^{*}\mathbf e_{c_1}=0\), we obtain
  \begin{equation*}
  \vec v_1^{*}\mat{G}^{(1)}\vec v_1
  =\sigma_1^{2}
  +\frac{\sigma_1^{2}}{|a_{r_1c_1}|^{2}}\,
  |w_{1,r_1}|^{2}\,\|\mat{A}\mathbf e_{c_1}\|_2^{2}.
  \end{equation*}
  Taking the conditional expectation over \(r_1\) given \(c_1\),
  \begin{align}
  \EX\!\bigl[\vec v_1^{*}\mat{G}^{(1)}\vec v_1 \mid c_1\bigr]
  &=\sigma_1^{2}
  +\sigma_1^{2}\sum_{i : a_{ic_1} \neq 0}^{n}
  \frac{|a_{ic_1}|^{2}}{\|\mat{A}\mathbf e_{c_1}\|_2^{2}}\cdot
  \frac{\|\mat{A}\mathbf e_{c_1}\|_2^{2}}{|a_{ic_1}|^{2}}\,|w_{1,i}|^{2}\notag\\
  &\leq\sigma_1^{2}+\sigma_1^{2}\sum_{i=1}^{n}|w_{1,i}|^{2}
  \;=\;2\sigma_1^{2}.\label{eq:case1-exp}
  \end{align}
  
  \paragraph{Case 2: \(\mat{P}\mathbf e_{c_1}\neq 0\).}
  Set
  \[
  \vec g
  =\frac{\mat{P}\mathbf e_{c_1}}{\|\mat{P}\mathbf e_{c_1}\|_2}.
  \]
  Expanding \(\vec g^* \mat{G}^{(1)} \vec g\) using \eqref{eq:C1-expanded}, and using that
  \[
  (\mat{P}\mathbf e_{c_1})^* \mat{A}^* \mat{A}\,\mathbf e_{c_1}
  = (\mat{P}\mathbf e_{c_1})^* \mat{A}^* \mat{A}\,\mat{P}\mathbf e_{c_1}
  = \|\mat{A}\mat{P}\mathbf e_{c_1}\|_2^2,
  \]
  we obtain the identity
  \begin{align}
  \frac{1}{\|\mat{P}\mathbf e_{c_1}\|_2^{2}}\,
  \mathbf e_{c_1}^{*}\mat{P}\,\mat{G}^{(1)}\,\mat{P}\mathbf e_{c_1}
  &=\frac{1}{|a_{r_1c_1}|^{2}\,\|\mat{P}\mathbf e_{c_1}\|_2^{2}}
  \Bigl(
  |a_{r_1c_1}|^{2}\,\|\mat{A}\mat{P}\mathbf e_{c_1}\|_2^{2}
  -2\|\mat{A}\mat{P}\mathbf e_{c_1}\|_2^{2} \Re\left({a_{r_1c_1}\,\,\overline{[\mat{A}\mat{P}]}_{r_1c_1}}\right)\notag\\
  &\hspace{6.7em}
  +\ \|\mat{A}\mathbf e_{c_1}\|_2^{2}\,|[\mat{A}\mat{P}]_{r_1c_1}|^{2}
  \Bigr) \\
  & = \Bigl( \|\mat{A}\mat{P}\mathbf e_{c_1}\|_2^2\big|a_{r_1c_1}\, - \,[\mat{A}\mat{P}]_{r_1c_1}\big|^{2}
  +\ \big|[\mat{A}\mat{P}]_{r_1c_1}\big|^{2}\,\|\mat{A}\mat{P}_{\perp}\mathbf e_{c_1}\|_2^{2}\Bigr) \\
  & = \Bigl( \|\mat{A}\mat{P}\mathbf e_{c_1}\|_2^2 \,|[\mat{A}\mat{P}_{\perp}]_{r_1c_1}|^{2}
  +\ |[\mat{A}\mat{P}]_{r_1c_1}|^{2}\,\|\mat{A}\mat{P}_{\perp}\mathbf e_{c_1}\|_2^{2}\Bigr)
  \label{eq:case2-pre}
  \end{align}
  Taking the conditional expectation over \(r_1\) given \(c_1\) and using \(\prob(r_1=i\mid c_1)=|a_{ic_1}|^{2}/\|\mat{A}\mathbf e_{c_1}\|_2^{2}\), we can rewrite \eqref{eq:case2-pre} as
  \begin{align}
  \EX\!\bigl[\vec g^{*}\mat{G}^{(1)}\vec g \mid c_1\bigr]
  &=\frac{1}{\|\mat{A}\mathbf e_{c_1}\|_2^{2}\,\|\mat{P}\mathbf e_{c_1}\|_2^{2}}
  \sum_{i:\,a_{ic_1}\neq 0}
  \Bigl( [\mat{A}\mat{P}_{\perp}]_{ic_1}^{2}\,\|\mat{A}\mat{P}\mathbf e_{c_1}\|_2^{2}
  + [\mat{A}\mat{P}]_{ic_1}^{2}\,\|\mat{A}\mat{P}_{\perp}\mathbf e_{c_1}\|_2^{2}\Bigr)\notag\\
  &\le
  \frac{2}{\|\mat{A}\mathbf e_{c_1}\|_2^{2}\,\|\mat{P}\mathbf e_{c_1}\|_2^{2}}\,
  \|\mat{A}\mat{P}_{\perp}\mathbf e_{c_1}\|_2^{2}\,\|\mat{A}\mat{P}\mathbf e_{c_1}\|_2^{2}\notag\\
  &\le
  2\sigma_1^{2}\,\frac{\|\mat{A}\mat{P}_{\perp}\mathbf e_{c_1}\|_2^{2}}{\|\mat{A}\mathbf e_{c_1}\|_2^{2}},
  \label{eq:case2-exp}
  \end{align}
  where in the last line we used \(\|\mat{A}x\|_2\le \sigma_1\|x\|_2\).
  
  Combining \eqref{eq:case1-exp}--\eqref{eq:case2-exp} and taking the expectation over \(c_1\) with \(\prob(c_1=j)=\|\mat{A}\mathbf e_j\|_2^{2}/\|\mat{A}\|_F^{2}\),
  \begin{align}
  \EX\bigl[\vec g^{*}\mat{G}^{(1)}\vec g\bigr]
  &\le
  \sum_{j:\,\mat{P}\mathbf e_{j}=0}\frac{\|\mat{A}\mathbf e_{j}\|_2^{2}}{\|\mat{A}\|_F^{2}}\cdot 2\sigma_1^{2}
  +
  \sum_{j:\,\mat{P}\mathbf e_{j}\neq 0}\frac{\|\mat{A}\mathbf e_{j}\|_2^{2}}{\|\mat{A}\|_F^{2}}\cdot
  2\sigma_1^{2}\,\frac{\|\mat{A}\mat{P}_{\perp}\mathbf e_{j}\|_2^{2}}{\|\mat{A}\mathbf e_{j}\|_2^{2}}\notag\\
  &\le
  \frac{2\sigma_1^{2}}{\|\mat{A}\|_F^{2}}\sum_{j}\|\mat{A}\mat{P}_{\perp}\mathbf e_{j}\|_2^{2}
  =
  \frac{2\sigma_1^{2}}{\|\mat{A}\|_F^{2}}\sum_{j=k+1}^{n}\sigma_j^{2}.
  \label{eq:g-term-final}
  \end{align}
  where we used that $\|\mat{A}e_{j}\|_2^2 =  \|\mat{A}\mat{P}_{\perp}e_{j}\|_2^2 + \|\mat{A}\mat{P}e_{j}\|_2^2$. Using \(\sigma_1^{2}\le \|\mat{A}\|_F^{2}\), we conclude from \eqref{eq:g-term-final} that
  \begin{equation}
  \EX\bigl[\vec g^{*}\mat{G}^{(1)}\vec g\bigr] \le 2\tr(\mat{G}-\lowrank{\mat{G}}_{k})
  \end{equation}

  Coming back to \eqref{eq:split-main}, we obtain:
  \[
  \EX\tr\!\bigl(\mat{G}^{(1)}-\lowrank{\mat{G}^{(1)}}_{k-1}\bigr)
  \;\le\;2\tr(\mat{G}-\lowrank{\mat{G}}_{k}) +
  2\tr(\mat{G}-\lowrank{\mat{G}}_{k})
  =
  4\,\tr\!\bigl(\mat{G}-\lowrank{\mat{G}}_{k}\bigr).
  \]
  Applying this one-step bound repeatedly,
  \begin{align*}
  \EX\|\mat{A}^{(k)}\|_{F}^{2}
  =\EX\tr(\mat{G}^{(k)})
  &\le 4\,\EX\tr\!\bigl(\mat{G}^{(k-1)}-\lowrank{\mat{G}^{(k-1)}}_{1}\bigr)\\
  &\le 4^{2}\,\EX\tr\!\bigl(\mat{G}^{(k-2)}-\lowrank{\mat{G}^{(k-2)}}_{2}\bigr)
  \;\le\;\cdots\\
  &\le 4^{k}\,\tr\!\bigl(\mat{G}-\lowrank{\mat{G}}_{k}\bigr)
  \;=\;
  4^{k}\,\|\mat{A}-\lowrank{\mat{A}}_{k}\|_{F}^{2},
  \end{align*}
  which is \cref{eq:doubling}.
  \end{proof}
  
  \subsection{Proof of~\cref{prop:SRPLU_vs_RPCholesky}} \label{sec:SRPLU_vs_RPCholesky_proof}
  \begin{proof}
    Taking expectations, we have
    \begin{equation} \label{eq:trace_diff_1}
    \tr\!\bigl(\EX[\mat{C}^{(1)}]\bigr) \;-\; \tr\!\bigl(\EX[\tilde{\mat{C}}^{(1)}]\bigr)
    =  \sum_{i,j}\frac{|c_{ij}|^{2}}{\|\mat{C}\|_F^2}\,\tr(\mat{K}_{ij}) \;-\; \frac{\|\mat{C}\|_F^2}{\tr(\mat{C})},
    \end{equation}
    where \(\mat{K}_{ij}\coloneqq \mat{L}_{ij}\,\mat{B}_{ij}^{\dagger}\,\mat{L}_{ij}^{*}\) (cf.~\cref{eq:SRPLU}).
    Define
    \[
      t_i \coloneqq
      \begin{cases}
        \|\mat{c}_i\|_2^2/c_{ii}, & c_{ii} \neq 0, \\
        0, & c_{ii} = 0,
      \end{cases}
    \]
    where \(\mat{c}_i\) denotes the \(i\)th column of \(\mat{C}\). Note that since \(\mat{C}\succeq 0\), the condition \(c_{ii}=0\) implies \(\|\mat{c}_i\|_2^2=0\). Furthermore, since \(\mat{K}_{ij}\succeq 0\), we have \(\tr(\mat{K}_{ij})\ge t_i\ge 0\) for all \(i,j\).
    Therefore
    \[
    \sum_{i,j}|c_{ij}|^{2}\,\tr(\mat{K}_{ij})
    \ge
    \sum_{i}t_i\!\sum_{j}|c_{ij}|^{2} \ge \sum_{i}c_{ii}\,t_i^{2}.
    \]
    where we have used that \(\sum_{j}|c_{ij}|^{2}=\|\mat{c}_i\|_2^{2} = c_{ii}t_i\).
    By Cauchy-Schwarz,
    \[
    \sum_{i}c_{ii}\,t_i^{2}
    \;\ge\;
    \frac{\bigl(\sum_{i}c_{ii}t_i\bigr)^{2}}{\sum_{i}c_{ii}}
    =\frac{\|\mat{C}\|_F^{4}}{\tr(\mat{C})}.
    \]
    Dividing by \(\|\mat{C}\|_F^{2}\), we get
    \[
     \sum_{i,j}\frac{|c_{ij}|^{2}}{\|\mat{C}\|_F^2}\,\tr(\mat{K}_{ij})
    \ge \frac{\|\mat{C}\|_F^{2}}{\tr(\mat{C})} \, 
    \]
    which shows that \eqref{eq:trace_diff_1} is non-negative, and hence proves \cref{eq:off-diag-2-diag-1}.
    \end{proof}

  \subsection{Proof of \cref{prop:trace_ratio}}
\label{sec:dynamics_proofs}

  Let $\mat{C}= \mat{U} \mat{\Lambda} \mat{U}^*$ be the eigendecomposition of $\mat{C}$, where $\mat{U} \in \mathbb{C}^{n \times n}$ is unitary and $\mat{\Lambda} = \diag(d_1, d_2, \dots, d_r, 0, \dots, 0)$ with $d_1\ge\cdots\ge d_r >0$ the non-zero eigenvalues of $\mat{C}$.
  Since the trace and the map $\Phi$ are invariant under unitary transformations, we have $\tr(\Phi^{\circ k}(\mat{C})) = \tr(\Phi^{\circ k}(\mat{\Lambda}))$.
  Let $\mat{D}^{(0)}=\mat{D}=\diag(d_1,\dots,d_r)$ be the $r \times r$ diagonal matrix of non-zero eigenvalues, and define
  \[
  \mat{D}^{(k+1)}=\Phi(\mat{D}^{(k)}),\qquad
  t_k=\tr\mat{D}^{(k)}=\sum_{i=1}^r d_i^{(k)},\qquad
  p_i^{(k)}=\frac{d_i^{(k)}}{t_k}\in(0,1).
  \]
  From the definition of $\Phi$, we have
  \[
  d_i^{(k+1)}=d_i^{(k)}\Bigl(1-\frac{d_i^{(k)}}{t_k}\Bigr)=t_k\,p_i^{(k)}\bigl(1-p_i^{(k)}\bigr),
  \qquad
  \frac{t_{k+1}}{t_k}=1-q_k,\ \ \text{where}\ \ q_k=\sum_{i=1}^r \bigl(p_i^{(k)}\bigr)^2.
  \]
  We will show that $p^{(k)} \to (\frac{1}{r}, \dots, \frac{1}{r})$, which implies that $\frac{t_{k+1}}{t_k} \to 1 - \frac{1}{r}$ and establishes~\cref{prop:trace_ratio}.
  
  Normalizing by $t_{k+1}=t_k(1-q_k)$ yields the induced map on the simplex
  \begin{equation}\label{eq:update}
  p_i^{(k+1)}=\frac{p_i^{(k)}(1-p_i^{(k)})}{1-q_k}.
  \end{equation}
  From \eqref{eq:update}, for any indices $i,j$,
  \begin{equation}\label{eq:pair}
  p_i^{(k+1)}-p_j^{(k+1)}
  =\frac{\bigl(p_i^{(k)}-p_j^{(k)}\bigr)\bigl(1-(p_i^{(k)}+p_j^{(k)})\bigr)}{1-q_k}.
  \end{equation}
  Since $p_i^{(k)}+p_j^{(k)}\le1$, the mapping is order preserving. Since the entries $d_i$ are decreasing, it is sufficient to show that $p^{(k)}_1  - p^{(k)}_r \to 0$.
  
  Let $M_k= p_1^{(k)}$ and $m_k= p_r^{(k)}$. By \eqref{eq:pair}:
  \begin{equation}\label{eq:osc-step}
  M_{k+1}-m_{k+1} = \frac{1-(M_k+m_k)}{1-q_k}\,(M_k-m_k).
  \end{equation}
  Note that
  \begin{equation}\label{eq:osc-step-1}
  \frac{1-(M_k+m_k)}{1-q_k} = 1 - \frac{(M_k+m_k)-q_k}{1-q_k} \leq 1 - \frac{m_k}{1-q_k} \le 1 - m_k 
  \end{equation}
  where we have used that $q_k=\sum_i\bigl(p_i^{(k)}\bigr)^2\ \le\ M_k\sum_i p_i^{(k)}=M_k$ and $1-q_k \le 1$.
  Furthermore, 
  $m_k$ is nondecreasing: \eqref{eq:update},
  \[
  m_{k+1}=\frac{m_k(1-m_k)}{1-q_k}
  \ \ge\ \frac{m_k(1-m_k)}{1-m_k}=m_k,
  \]
  where we have used that $q_k=\sum_i\bigl(p_i^{(k)}\bigr)^2\ \ge\ m_k\sum_i p_i^{(k)}=m_k$.
  Hence $1 - m_0 \le 1 - m_k$ for all $k \ge 0$ and, from \eqref{eq:osc-step} and \eqref{eq:osc-step-1},
  \[
  M_{k+1}-m_{k+1}
  \ \le\ (1-m_k)\,(M_k-m_k)
  \ \le\ \prod_{i=0}^{k} (1-m_i) (M_0-m_0)
  \ \le\ (1-m_0)^{k+1} (M_0-m_0)
  \]
  which implies that $M_k - m_k \to 0$.

\section{Algorithms}

\begin{algorithm}[H]
  \caption{Q-less truncated pivoted QR~\cite{stewart1999four}}
  \label{alg:pivoted_qr}
  \begin{algorithmic}
    \State \textbf{Input:} Matrix $\mat{A} \in \mathbb{R}^{n \times m}$, target rank $r$, initial column norms $\bm{c} = [\|\mat{A}_{:,j}\|_2^2]_{j=1}^m$
    \State \textbf{Output:} Upper triangular factor $\mat{R} \in \mathbb{R}^{r \times r}$, selected column indices $\J_r$
    \State Initialize $\J_0 \leftarrow \emptyset$, $\mat{R} \leftarrow \mat{0}_{r \times r}$
    \For{$k = 1$ to $r$}
      \State \textbf{Step 1: Select pivot column}
      \State Pick $j_k = \argmax_{j \notin \J_{k-1}} \bm{c}_j$ (or sample with $\prob(j_k=j) = \bm{c}_{j} / \sum_{j'} \bm{c}_{j'}$)
      \State $\J_k \leftarrow \J_{k-1} \cup \{j_k\}$
      \State $u \leftarrow \mat{A}_{:,j_k}$

      \State \textbf{Step 2: Gram--Schmidt step $u - Q_{k-1}Q_{k-1}^T u$ using $Q_{k-1} = A_{:,\J_{k-1}} R^{-1}_{1:k-1,1:k-1}$}
      \State $v \leftarrow (\mat{A}_{:,\J_{k-1}})^T u$
      \State Solve $\mat{R}_{1:k-1, 1:k-1}^T y = v$
      \State Solve $\mat{R}_{1:k-1, 1:k-1} w = y$
      \State $u \leftarrow u - \mat{A}_{:,\J_{k-1}} w$
      \State $\mat{R}_{1:k-1, k} \leftarrow y$
      \State \textbf{Step 3: Gram-Schmidt (second pass)}
      \State $v \leftarrow (\mat{A}_{:,\J_{k-1}})^T u$
      \State Solve $\mat{R}_{1:k-1, 1:k-1}^T y = v$
      \State Solve $\mat{R}_{1:k-1, 1:k-1} w = y$
      \State $u \leftarrow u - \mat{A}_{:,\J_{k-1}} w$
      \State $\mat{R}_{1:k-1, k} \leftarrow \mat{R}_{1:k-1, k} + y$
      \State $\mat{R}_{k,k} \leftarrow \|u\|_2$
      \State $u \leftarrow u / \mat{R}_{k,k}$

      \State \textbf{Step 4: Downdate column norms}
      \State $g \leftarrow \mat{A}^T u$
      \State $\bm{c} \leftarrow \bm{c} - |g|^{\odot 2} $
    \EndFor
  \end{algorithmic}
\end{algorithm}

\begin{algorithm}[H]
  \caption{QR-based CUR decomposition}
  \label{alg:qr_cur}
  \begin{algorithmic}
    \State \textbf{Input:} Matrix $\mat{A} \in \mathbb{R}^{n \times m}$, target rank $r$
    \State \textbf{Output:} $\I$, $\J$, $\mat{U}$
    \State $[\mat{R}_1, \J] \leftarrow \text{QR}(\mat{A}, r )$ \Comment{Select columns $\J$ via pivoted QR}
    \State $[\mat{R}_2, \I] \leftarrow \text{QR}(\mat{A}^T, r )$ \Comment{Select rows $\I$ via pivoted QR on $\mat{A}^T$}
    \State Assemble $\mat{Z} = \mat{A}_{:,\J}^T \mat{A} \mat{A}_{\I,:}^T$ using matvecs
    \State Compute $\mat{U} =(\mat{R}_1^T \mat{R}_1)^{-1} \mat{Z} (\mat{R}_2^T \mat{R}_2)^{-1}$ using triangular solves
  \end{algorithmic}
\end{algorithm}

\begin{algorithm}[H]
  \caption{AAA algorithm for rational approximation}
  \label{alg:fast-rational-approximation}
  \begin{algorithmic}
    \State \textbf{Input:} sample points $\mathbf{Z}=(z_i)_{i=1}^m$, values $\mathbf{F}=(f_i)_{i=1}^m$, tolerance $\epsilon$
    \State Initialize $\mathcal{S}\leftarrow \emptyset$, $r \leftarrow 0$
    \For{$k = 1,2,3,\ldots$}
      \State $t_k \leftarrow \argmax_{z_i \in \mathbf{Z}\setminus \mathcal{S}} |f_i - r(z_i)|$
      \State $\mathcal{S} \leftarrow \mathcal{S}\cup\{t_k\}$
      \State Form $\mat{L}$ with entries $\mat{L}_{i,\ell}=\frac{f(z_i)-f(t_\ell)}{z_i-t_\ell}$ for $z_i\in \mathbf{Z}\setminus\mathcal{S}$
      \State $w \leftarrow$ smallest right singular vector of $\mat{L}$
      \State Define $r$ by the barycentric form \eqref{eq:aaa-barycentric}
      \State \textbf{Terminate} if $\max_{1\le i\le m} |f_i - r(z_i)| \le \epsilon$
    \EndFor
    \State \textbf{Output:} rational function $r$
  \end{algorithmic}
\end{algorithm}

\section{Row norm bounds for Cauchy-like matrices}\label{sec:cauchy-row-norm-bound}
We describe a tree-based algorithm for computing upper bounds of row norms for Cauchy-like matrices, similar to a Barnes-Hut algorithm.

Let $\mat{A}\in\mathbb{C}^{n\times m}$ be Cauchy-like with displacement rank $p$:
\begin{equation}\label{eq:gbub-cauchy-like}
a_{i,j} \;=\; \sum_{\ell=1}^p \frac{g_{i,\ell}\,b_{\ell,j}}{x_i-y_j},
\qquad x_i\neq y_j \ \ \forall i,j,
\end{equation}
where $\mat{G}=(g_{i,\ell})\in\mathbb{C}^{n\times p}$ and $\mat{B}=(b_{\ell,j})\in\mathbb{C}^{p\times m}$ are the generators.
Our goal is to compute an upper bound vector $u\in\mathbb{R}^n$ such that, for some constant $\nu \ge 1$,
\begin{equation}\label{eq:gbub-goal}
\norm{\mat{A}_{i,:}}_2^2 \;\le\; u_i \;\le\; \nu\,\norm{\mat{A}_{i,:}}_2^2
\qquad \forall i,
\end{equation}
where the squared row norm is given by
\[
\|\mat{A}_{i,:}\|_2^2 = \sum_{j = 1}^m \frac{\bigl|\mat{G}_{i,:}\mat{B}_{:,j}\bigr|^2}{|x_i - y_j|^2}.
\]
We refer to the points $x_i$ as \emph{target points} and $y_j$ as \emph{source points}.
For subsets $S \subseteq \{1, \ldots, m\}$ and $T \subseteq \{1, \ldots, n\}$, define the minimum and maximum squared distances:
\[
d_{\min}^2(S,T) = \min_{i \in T, j \in S} |x_i - y_j|^2, \qquad
d_{\max}^2(S,T) = \max_{i \in T, j \in S} |x_i - y_j|^2,
\]
and the Gram matrix
\[
\mat{H}_S \;\coloneqq\; \sum_{j\in S} \mat{B}_{:,j}\mat{B}_{:,j}^\ast \ \in\ \mathbb{C}^{p\times p}.
\]

To illustrate the key idea, suppose a set of source points $S$ satisfies
\[
d_{\max}^2(S,\{i\}) \leq \nu d_{\min}^2(S,\{i\})
\]
for a target point $i$.
Then the corresponding terms in the sum are bounded by:
\[
\sum_{j\in S}\frac{\bigl|\mat{G}_{i,:}\mat{B}_{:,j}\bigr|^2}{|x_i-y_j|^2}
\le
\frac{\mat{G}_{i,:}\,\mat{H}_S\,\mat{G}_{i,:}^\ast}{d_{\min}^2(S,\{i\})}
\le 
\nu \frac{\mat{G}_{i,:}\,\mat{H}_S\,\mat{G}_{i,:}^\ast}{d_{\max}^2(S,\{i\})}
\le \nu \sum_{j\in S}\frac{\bigl|\mat{G}_{i,:}\mat{B}_{:,j}\bigr|^2}{|x_i-y_j|^2},
\]
which establishes~\cref{eq:gbub-goal} for the contribution from source points in $S$.

To compute the full sum over all target points, we partition the source and target points using two quadtrees and construct interaction lists.
The algorithm in~\cref{alg:gbub-precompute} constructs these quadtrees and, for each target leaf $\ell$, maintains two lists:
\begin{itemize}
  \item $\mathcal{A}_\ell$: a list of aggregated entries $(s,\alpha)$, where $s$ is a source node and $\alpha = 1/d_{\min}^2(s,\ell)$.
  \item $\mathcal{E}_\ell$: a list of source leaf nodes $s$ that must be handled exactly (not aggregated).
\end{itemize}

\begin{algorithm}[H]
\caption{Precompute interaction lists}\label{alg:gbub-precompute}
\begin{algorithmic}[1]
\State \textbf{Input:} target points $\{x_i\}_{i=1}^n$, source points $\{y_j\}_{j=1}^m$, factor $\nu\ge 1$, leaf size $L$
\State \textbf{Output:} quadtrees $\mathcal{T}_{tgt}$, $\mathcal{T}_{src}$ and interaction lists $(\mathcal{A}_\ell,\mathcal{E}_\ell)$ for each target leaf $\ell$
\State Build target quadtree $\mathcal{T}_{tgt}$ on $\{x_i\}$ with max leaf size $L$
\State Build source quadtree $\mathcal{T}_{src}$ on $\{y_j\}$ with max leaf size $L$
\State Initialize $(\mathcal{A}_\ell,\mathcal{E}_\ell)\gets(\emptyset,\emptyset)$ for every target leaf $\ell$
\State Initialize stack $\mathcal{S}\gets\{(s_0,t_0)\}$ with root nodes of $(\mathcal{T}_{src},\mathcal{T}_{tgt})$
\While{$\mathcal{S}$ not empty}
  \State Pop $(s,t)$ from $\mathcal{S}$
  \State Compute $d_{\min}^2(s,t)$ and $d_{\max}^2(s,t)$ from the two bounding boxes
  \If{ $d_{\max}^2(s,t) \le \nu\,d_{\min}^2(s,t)$}
    \For{each target leaf $\ell$ contained in node $t$}
      \State Add $(s,\alpha=1/d_{\min}^2(s,\ell))$ to $\mathcal{A}_\ell$
    \EndFor
    \State \textbf{continue}
  \EndIf
  \If{$s$ is a leaf \textbf{and} $t$ is a leaf}
    \State Let $\ell$ be the target leaf corresponding to $t$
    \State Add $s$ to $\mathcal{E}_\ell$
    \State \textbf{continue}
  \EndIf
  \If{size$(t) \ge$ size$(s)$}
    \For{each child $t'$ of $t$} \State Push $(s,t')$ to $\mathcal{S}$ \EndFor
  \Else
    \For{each child $s'$ of $s$} \State Push $(s',t)$ to $\mathcal{S}$ \EndFor
  \EndIf
\EndWhile
\end{algorithmic}
\end{algorithm}

An important feature of this approach is that the quadtrees and interaction lists are computed once during preprocessing and can be reused across all iterations of RPLU. In numerical experiments, we observe that the online phase is typically $5$--$20\times$ faster than the precomputation phase.

Given the precomputed interaction lists and generators $\mat{G},\mat{B}$, the online phase proceeds as follows.
First, compute the Gram matrices $\mat{H}_s \coloneqq \sum_{j\in s}\mat{B}_{:,j}\mat{B}_{:,j}^\ast$ for all source nodes $s$ via a bottom-up traversal of the source quadtree.
Then, for each target leaf $\ell$ and each target point $i\in \ell$, the upper bound is given by
\[
u_i \;=\;
\sum_{(s,\alpha)\in\mathcal{A}_\ell} \alpha\ \mat{G}_{i,:}\mat{H}_s\mat{G}_{i,:}^\ast
\;+\;
\sum_{s\in\mathcal{E}_\ell}\ \sum_{j\in s}\ \frac{\bigl|\mat{G}_{i,:}\mat{B}_{:,j}\bigr|^2}{|x_i-y_j|^2}.
\]

\begin{algorithm}[H]
\caption{Compute certified upper bounds}\label{alg:gbub-online}
\begin{algorithmic}[1]
\State \textbf{Input:} precomputed plan $(\mathcal{A}_\ell,\mathcal{E}_\ell)$, generators $\mat{G}\in\mathbb{C}^{n\times p}$, $\mat{B}\in\mathbb{C}^{p\times m}$
\State \textbf{Output:} $u\in\mathbb{R}^n$ satisfying~\eqref{eq:gbub-goal}
\State Compute Gram matrices $\mat{H}_s \coloneqq \sum_{j\in s}\mat{B}_{:,j}\mat{B}_{:,j}^\ast$ for all source nodes $s$ via bottom-up traversal
\State Initialize $u \gets 0_n$
\For{each target leaf $\ell$}
  \For{each $(s,\alpha)\in\mathcal{A}_\ell$}
    \For{each $i\in \ell$}
      \State $u_i \gets u_i + \alpha\ \mat{G}_{i,:}\mat{H}_s\mat{G}_{i,:}^\ast$
    \EndFor
  \EndFor
  \For{each source leaf $s\in\mathcal{E}_\ell$}
    \For{each $j\in s$}
      \For{each $i\in \ell$}
        \State $u_i \gets u_i + \dfrac{\bigl|\mat{G}_{i,:}\mat{B}_{:,j}\bigr|^2}{|x_i-y_j|^2}$
      \EndFor
    \EndFor
  \EndFor
\EndFor
\State \Return $u$
\end{algorithmic}
\end{algorithm}

\bibliographystyle{alpha}
\bibliography{sample}

\end{document}